\documentclass{amsart}
\usepackage[utf8]{inputenc}
\usepackage{a4wide}
\usepackage{amsmath,amsfonts,amssymb}
\usepackage{amsthm}
\usepackage{ifthen}
\usepackage{longtable}
\usepackage{array}
\usepackage{url}
\usepackage{enumerate}
\usepackage{xcolor}
\usepackage{hyperref}

\usepackage{titlesec}
\titleformat{\section}
  {\center\normalsize\bfseries}{\thesection.}{1em}{}
\titleformat{\subsection}
  {\normalsize\bfseries}{\thesubsection.}{1ex}{}

\hypersetup{colorlinks,
	linkcolor=green!50!black,
	citecolor=red!40!black,
	urlcolor=blue!40!black
}

\newcommand{\Z}{\mathbb{Z}}
\newcommand{\Q}{\mathbb{Q}}
\newcommand{\R}{\mathbb{R}}
\newcommand{\C}{\mathbb{C}}

\DeclareMathOperator{\h}{h}

\newcommand{\bsubtr}{b_{\text{subtr}}}
\newcommand{\hsubtr}{h_{\text{subtr}}}

\newcommand{\caselIold}[1]{\medskip\noindent\textit{#1}}
\newcommand{\caselI}[1]{\medskip\noindent\textbf{#1}}
\newcommand{\caselII}[1]{\smallskip\textbf{#1}}
\newcommand{\caselIII}[1]{\smallskip\textbf{\indent #1}}

\newtheorem*{theorem*}{Theorem}
\newtheorem{knownthm}{Theorem}[section]
\newtheorem{newlem}[knownthm]{Lemma}
\newtheorem{knownlem}[knownthm]{Lemma}
\newtheorem{cor}[knownthm]{Corollary}

\newtheorem*{claim*}{Claim}
\newtheorem{knownprop}[knownthm]{Proposition}
\newtheorem{newthm}[knownthm]{Theorem}
\newtheorem{newprop}[knownthm]{Proposition}

\theoremstyle{definition}
	\newtheorem{defi}[knownthm]{Definition}
	
	\newtheorem{knownquest}[knownthm]{Question}
	\newtheorem{rem}[knownthm]{Remark}
	\newtheorem*{rem*}{Remark}
	\newtheorem{knownprobl}[knownthm]{Problem}
	
	\newtheorem{newquest}[knownthm]{Question}

\numberwithin{equation}{section}

\begin{document}

\title[Consecutive multiplicatively dependent triples]{More on consecutive multiplicatively dependent triples of integers}
\thanks{M.B.\ was supported by a grant from NSERC. I.P. was supported by the NKFIH grant ANN 130909. I.V.\ was supported by the Austrian Science Fund (FWF) under the project I4406, as well as by the Austrian Federal Ministry of Education, Science and Research (BMBWF), grant no. SPA 01-080 MAJA}

\author[M. A. Bennett]{Michael A. Bennett}
\address{M. A. Bennett, University of British Columbia, Department of Mathematics, 1984 Mathematics Road, Vancouver B.C., Canada V6T 1Z2}
\email{bennett@math.ubc.ca}

\author[I. Pink]{István Pink}
\address{I. Pink,
University of Debrecen, Institute of Mathematics,
H-4002 Debrecen, P.O. Box 400,
Hungary}
\email{pinki@science.unideb.hu}

\author[I. Vukusic]{Ingrid Vukusic}
\address{I. Vukusic,
University of Waterloo,
Faculty of Mathematics,
200 University Avenue West,
Waterloo Ontario,
Canada N2L 3G1}
\email{ingrid.vukusic@uwaterloo.ca}

\begin{abstract}
In this paper, we extend recent work of the third author and Ziegler on triples of integers $(a,b,c)$, with the property that each of $(a,b,c)$, $(a+1,b+1,c+1)$ and $(a+2,b+2,c+2)$ is {\it multiplicatively dependent}, completely classifying such triples in case $a=2$. Our techniques include a variety of elementary arguments together with more involved machinery from Diophantine approximation.
\end{abstract}

\maketitle

\section{Introduction}

An $n$-tuple of numbers $(z_1, \ldots, z_n) \in \C^n$ is said to be \textit{multiplicatively dependent}, if there exists a non-zero tuple $(k_1, \ldots, k_n) \in \Z^n$ such that 
\[
	z_1^{k_1} \cdots z_n^{k_n} = 1. 
\]
In this paper, we will consider tuples $(z_1, \ldots, z_n)$ of positive integers and, for simplicity,  assume $1 < z_1 < \dots < z_n$. Of course, there exist many multiplicatively dependent tuples $(z_1, \ldots, z_n)$; see \cite{PappalardiShaShparlinskiStewart2018} for asymptotics.
We will focus our attention on ``consecutive'' tuples, by which we mean pairs of tuples of the shape $(z_1, \ldots, z_n)$ and $(z_1 + 1, \ldots, z_n + 1)$.

For pairs of integers, the situation is rather simple: If $1<a<b$, then $(a,b)$ is multiplicative dependent if and only if it is of the shape $(a,b) = (q^x, q^y)$ for some integers $q>1$ and $1 \leq x<y$.
Thus, it is relatively easy to see, appealing for example to Mih\u{a}ilescu's theorem  \cite{Mihailescu2004}, that $(2,8)$ and $(3,9)$ are the only consecutive multiplicatively dependent pairs of integers. In fact, this statement is equivalent to a well-known theorem of LeVeque \cite{LeVeque1952} from the 1950s; see \cite{VukusicZiegler2021} for details.

For triples, the situation is significantly  more complicated. 
From the definition of multiplicative dependence, a triple can be multiplicatively dependent even if only two of the numbers ``contribute'' to the dependence, e.g.\ $(2,8,13)$ is multiplicatively dependent, and so is $(3,9,14)$. From this observation, it is very easy to generate integer triples such that both $(a,b,c)$ and $(a+1,b+1,c+1)$ are multiplicatively dependent. Even if we exclude $a=2$ and $b=8$, one can take triples of the shape $(a,b,c)=(q^x, q^y, (q^x + 1)^s - 1)$; see \cite[Theorem 4]{VukusicZiegler2021} for more examples. 

In fact, one can even construct a family of three consecutive multiplicatively dependent triples:
\begin{equation}\label{eq:fam1}
	(2,8,2^x 5^y - 2), (3,9,2^x 5^y - 1), (4,10,2^x 5^y),
\end{equation}
where $x$ and $y$ are nonzero integers.
It is not well understood yet how many such triples exist outside of this family. In the range $1 < a < b < c \leq 1000$, we find  $11$ triples $(a,b,c)$ such that $(a,b,c)$, $(a+1,b+1,c+1)$ and $(a+2,b+2,c+2)$ are each multiplicatively dependent, and that are not in the family \eqref{eq:fam1}.
In \cite[Question 3]{VukusicZiegler2021} the authors posed the following question.

\begin{knownquest}\label{quest:knownYes}
Are there infinitely many triples $(a,b,c)$ of pairwise distinct integers larger than 1 with $\{2,8\}\not\subset \{a,b,c\}$ such that $(a,b,c)$, $(a+1,b+1,c+1)$ and $(a+2,b+2,c+2)$ are each multiplicatively dependent?
\end{knownquest}

This question turns out to be easier to answer than the authors of \cite{VukusicZiegler2021} expected. Indeed, the answer is ``yes'', as  each of the three consecutive triples 
\begin{align}
&(2, 2^x-2, 2^{2x}-2^{x+1}) = (2, 2(2^{x-1}-1), 2^{x+1}(2^{x-1}-1)), \nonumber\\
&(3, 2^x-1, 2^{2x}-2^{x+1}+1) = (3, 2^x - 1, (2^x-1)^2),\label{eq:newFamily} \\
&(4, 2^x, 2^{2x}-2^{x+1}+2) = (2^2, 2^x, (2^x-1)^2 + 1) \nonumber
\end{align}
is clearly multiplicatively dependent for any integer $x\geq 3$.

In particular, this means that, for fixed $a=2$, there exist infinitely many integers $2<b<c$ such that $(a,b,c)$, $(a+1,b+1,c+1)$ and $(a+2,b+2,c+2)$ are each multiplicatively dependent. Interestingly, the main result in \cite{VukusicZiegler2021} ensures that this cannot happen for any $a \neq 2$:

\begin{knownthm}[{cf.\ \cite[Theorem 5]{VukusicZiegler2021}}]\label{thm:oldMain_newVer}
Let $a\geq 3$ be a fixed integer. Then there are only finitely many integers $b,c$ with $a<b<c$ such that $(a,b,c)$, $(a+1,b+1,c+1)$ and $(a+2,b+2,c+2)$ are each multiplicatively dependent, and they are bounded effectively in terms of $a$.
\end{knownthm}

Note that Theorem 5 in \cite{VukusicZiegler2021} is formulated slightly differently:
there, the triples are not ordered and $a \notin \{2,8\}$ is required. However, the proof reveals that $a=8$
only needs to be excluded if $b=2$ or $c=2$ is allowed, which in the above formulation is not the case.

In view of Theorem~\ref{thm:oldMain_newVer}, an obvious question is the following: For $a=2$, how many triples outside of the families \eqref{eq:fam1} and \eqref{eq:newFamily} are there? We answer this question completely with the next result.

\begin{newthm}\label{thm:a2}
Let $(2,b,c)$ be a triple of integers with $2<b<c$, and the property that $(2,b,c)$, $(3,b+1,c+1)$ and $(4,b+2,c+2)$ are each multiplicatively dependent. Then
either $(a,b,c)$ is in one of the families \eqref{eq:fam1}, \eqref{eq:newFamily}, or
$(a,b,c)=(2,4,14)$, or $(a,b,c)=(2,6,16)$.
\end{newthm}

The proof will be provided in Section \ref{sec:a2}. It relies on known results on certain Diophantine equations, as well as on the two equations
\begin{equation}\label{eq:myeqs}
	2^y ( 2^x \pm 1)^z \mp 1
	= 3^{r} (2^{x+1} \pm 1)^{w},
\end{equation}
which are solved in Section~\ref{sec:myeqs} using lower bounds for linear forms in logarithms.

For now, let us return to Question \ref{quest:knownYes}. In view of family \eqref{eq:newFamily}, it makes sense to modify it in the following way.

\begin{newquest}\label{quest:new}
Do there exist infinitely many integer triples $(a,b,c)$ with $1<a<b<c$, $(a,b)\neq (2,8)$, and $(a,b,c)\neq (2, 2^x-2, 2^{2x}-2^{x+1})$, such that $(a,b,c)$, $(a+1,b+1,c+1)$ and $(a+2,b+2,c+2)$ are each multiplicatively dependent?
\end{newquest}

We will not be able to answer this question completely. However, let us have a look at the triples in the range $1<a<b<c\leq 1000$ that are included in Question~\ref{quest:new}. A quick computer search reveals the following seven triples $(a,b,c)$:
\begin{equation}\label{eq:smalltriples}
	(2, 4, 14), (3, 6, 48), (6, 8, 48), (6, 18, 48), (6, 30, 216), (7, 15, 49), (7, 49, 79), (8, 32, 98).
\end{equation}
Note that the triples $(2, 4, 14), (6, 30, 216),  (7, 15, 49), (7, 49, 79), (8, 32, 98)$ all have the property that only two of the entries contribute to the multiplicative dependence. We call such triples 2-multiplicatively dependent. A more common way to express this might be to speak about the multiplicative rank of the triple;  for simplicity, we use the following terminology, as introduced in \cite{VukusicZiegler2021}.

\begin{defi}
We say that an $n$-tuple $(a_1, \ldots , a_n)$ is $k$-multiplicatively dependent if there is a multiplicatively dependent 
$k$-subtuple 
$(a_{i_1}, \ldots , a_{i_k})$ and every 
$(k-1)$-subtuple $(a_{j_1}, \ldots , a_{j_{k-1}})$
is not multiplicatively dependent.
\end{defi}

Examining the first triple from \eqref{eq:smalltriples}, we observe that $(2,4,14)$ is 2-multiplicatively dependent, while its two associated triples,  $(3,5,15)$ and $(4,6,16)$, are  3-multiplicatively dependent\ and  2-multiplicatively dependent, respectively. In fact, for each triple $(a,b,c)$ in \eqref{eq:smalltriples}, we always have a similar ``mixed situation''. In other words, in our examples, $(a,b,c)$, $(a+1,b+1,c+1)$ and $(a+2,b+2,c+2)$ are never all 2-multiplicatively dependent\ or all 3-multiplicatively dependent While we don't know how to find or prove the absence of three consecutive 3-multiplicatively dependent\ triples, the question about three consecutive 2-multiplicatively dependent\ triples translates to tractable Diophantine equations. In fact, in \cite[Problem 2]{VukusicZiegler2021}, the following problem was posed.

\begin{knownprobl}\label{probl:solveEquations}
Solve the Diophantine equations
\begin{align*}
	((d^x + 1)^s + 1)^t - d^y  &= 2,\\
	(d^x + 1)^s - (d^y - 1)^m  &= 2,\\
	d^y - ((d^x - 1)^s - 1)^t  &= 2,
\end{align*}
and thus find all integers $1<a<b<c$ such that $(a,b,c)$, $(a+1,b+1,c+1)$ and $(a+2,b+2,c+2)$ are each 2-multiplicatively dependent
\end{knownprobl}

This is exactly what we will do in the second part of this paper (Section~\ref{sec:problem_equations} and \ref{sec:proof_3x2md}).

\begin{newthm}\label{thm:3x2md}
Let $(a,b,c)$ be a triple of integers with $1<a<b<c$, and the property that $(a,b,c)$, $(a+1,b+1,c+1)$ and $(a+2,b+2,c+2)$ are each 2-multiplicatively dependent. Then
\begin{equation}\label{eq:3x2-md}
	(a,b,c) \in \{
		(2,6,8),
		(2, 8, 2^x - 2),
		(2,8,10^y - 2)
	\},
\end{equation}
where $x \geq 4$ and $y \geq 2$ are integers.
\end{newthm}

As an ingredient in our proof, we will need rather strong bounds on the variables of the Diophantine equation 
\[
	x^2 - 2 = y^n. 
\]
This equation has been studied intensively, but is still not solved completely. We will collect some results and derive a small bound on the exponent $n$ in Section \ref{sec:x2-2} through application of Laurent's lower bound for linear forms in two logarithms.

In the next section, we collect some preliminary results. 
In Section~\ref{sec:myeqs}, we solve the two equations \eqref{eq:myeqs}, which we then use to prove Theorem~\ref{thm:a2} in Section~\ref{sec:a2}.
The equation $x^2 - 2 = y^n$ is investigated in Section~\ref{sec:x2-2}, and we solve the equations from Problem~\ref{probl:solveEquations} in Section~\ref{sec:problem_equations}. Finally, in Section~\ref{sec:proof_3x2md}, we prove Theorem~\ref{thm:3x2md}.

\section{Preliminaries}\label{sec:prelim}

Here, we collect various known results, which will be used later in the paper. We start with two results on consecutive multiplicatively dependent pairs of integers, then list several results on certain Diophantine equations and  recall a number of elementary lemmata related to divisibility. Finally, we state some lower bounds for linear forms in logarithms.

\subsection{Consecutive multiplicatively dependent pairs}

\begin{knownlem}[{\cite[Theorem 1]{VukusicZiegler2021}}]\label{lem:consecutivePairs}
The only integers $1 < a < b$, such that $(a, b)$ and $(a +1, b +1)$
are both multiplicatively dependent
are $a=2$ and $b=8$.
\end{knownlem}

\begin{knownlem}[{\cite[Theorem 3]{VukusicZiegler2021}}]\label{lem:consecutivePairs2}
There exist no integers $1 < a < b$, such that $(a, b)$ and $(a +2, b +2)$
are both multiplicatively dependent.
\end{knownlem}

\subsection{Some Diophantine equations}

\begin{knownthm}[Mih\u{a}ilescu, 2002  \cite{Mihailescu2004}]\label{thm:Catalan}
The only solution to the equation
\[
	a^x - b^y = 1,
	\quad \text{with } a,x,b,y \in \Z_{\geq 2}
\]
is $(a,x,b,y) = (3,2,2,3)$.
In other words, the only consecutive perfect powers are 8 and 9.
\end{knownthm}


\begin{knownthm}[Nagell, 1954 \cite{Nagell1954}]\label{thm:Nagell}
The equation
\[
	x^n - y^2 = 2
\]
has exactly one solution in positive integers $x,y$ and $n\geq 2$, which is given by $(x,y,n)=(3,5,3)$.
\end{knownthm}

\begin{knownthm}[Bennett, 2001 {\cite[Theorem 1.5]{Bennett2001}}]\label{thm:bennett2001-t2}
For any fixed integers $a\geq 2,b\geq 2$ the equation
\begin{equation}\label{eq:bennet2001}
	a^x - b^y = 2
\end{equation} 
has at most one solution $(x,y)$ in positive integers.
\end{knownthm}

\begin{knownlem}[Ljunggren, 1942 \cite{Ljunggren1942}, {St\o rmer, 1899 \cite[p.\ 168]{Stormer1899}}]\label{lem:Stormer}
The only solution to the equation
\[
	x^2 + 1 = 2y^n
\]
in integers $x \geq 2$, $y\geq 1$ and $n\geq 3$ is given by $(x,y,n) = (239,13,4)$.
\end{knownlem}

\begin{knownlem}\label{lem:two_squares_minus_one}
The only solution to the equation
\[
	2x^2 - 1 = y^n
\]
in integers $x\geq 2$, $y\geq 2$, $n\geq 3$ is $(x,y,n) = (78, 23, 3)$.
\end{knownlem}
\begin{proof}
From \cite[Theorem 1.1]{BennettSkinner2004} (with $C=2$ and the restriction $y=1$) it follows that there are no solutions with $n\geq 4$. For $n=3$, we obtain from \cite[Theorem 2]{Cohn1991} (with $N = 2$) the single solution $(x,y,n) = (78, 23, 3)$.
\end{proof}

\begin{knownlem}\label{lem:23q}
The equation
\begin{equation*}\label{eq:23q}
	2^k - 3^\ell x^n = \pm 1
\end{equation*}
has no solution in integers 
$k \geq 0$, $\ell \geq 0$, $x \geq 2$ and $n \geq 3$.
\end{knownlem}
\begin{proof}
This follows from  \cite[Theorem~1.2]{Bennett2004} with $a=3^\ell$, $b=2^k$, and the restriction $y = 1$.
\end{proof}

\begin{knownlem}\label{lem:powerOf3minus2perfectpowers}
The only solution to the equation
\begin{equation}\label{eq:powerOf3minus2perfectpowers}
	3^k - 2 x^n = \pm 1
\end{equation}
in integers $k \geq 1$, $x \geq 3$, $n \geq 2$
is given by $(k,x,n) = (5, 11, 2)$, where we have
\[
	3^5 - 2 \cdot 11^2 = 1.
\]
\end{knownlem}
\begin{proof}

For $n\geq 3$ the statement follows immediately from \cite[Theorem~1.2]{Bennett2004} with $a = 2$ and $b = 3^k$ and the restriction $y = 1$.

Now let $n = 2$. Considering equation \eqref{eq:powerOf3minus2perfectpowers} modulo 3, we see that the right hand side has to be $+1$, i.e.\ we have the equation $3^k - 1 = 2 x^2$. 
If $k$ is odd, then \cite[Theorem 1.1]{BennettSkinner2004} (with $C = 2$, $x = 3$, $y = -1$) implies that $k\leq 5$. 
If $k$ is even, then  
$$
(3^{k/2}-1)(3^{k/2}+1) = 2  x^2. 
$$
Since $3^{k/2}-1, 3^{k/2}+1$ have no common prime factors $>2$, we get that $3^{k/2}+1$ is either a square or two times a square. If $3^{k/2}+1$ is a square, then, via Theorem \ref{thm:Catalan},  $k \leq 2$. If $3^{k/2}+1$ is twice a square, we obtain a contradiction modulo 3.
It follows that  $k \in \{ 2,3,5 \}$, whence a short check reveals that $k=5$.
\end{proof}

\subsection{Some divisibility lemmata}

The first two lemmata are stated without reference, being  standard and easily verified.

\begin{knownlem}\label{lem:qx-1_qy-1}
Let $a,m,n \in \Z_{\geq 1}$. Then
$a^m-1$ divides $a^n-1$ if and only if $m$ divides $n$.
\end{knownlem}

\begin{knownlem}\label{lem:divis_axplus1}
Let $a \in \Z_{\geq 2}$ and $m,n \in \Z_{\geq 1}$. If
$a^m+1$ divides $a^n+1$, then $m$ divides $n$.
\end{knownlem}

If $m$ is a nonzero integer, let $v_p(m)$ denote the largest integer $k$ with the property that $p^k \mid m$.

\begin{knownlem}[{see e.g.\ \cite[Lemma 2.1.22]{Cohen2007I}}]\label{lem:vp-1}
Let $p>2$ be prime and $a$ an integer with $p \mid a-1$. Then for any positive integer $n$ we have
\[
	v_p(a^n - 1) = v_p(a-1) + v_p(n).
\]
\end{knownlem}

\begin{knownlem}\label{lem:vp+1}
Let $p>2$ be prime and $a$ an integer with $p \mid a+1$. Then for any positive integer $n$ we have
\[
	v_p(a^n + 1) = \begin{cases}
	v_p(n) + v_p(a+1), & \text{ if }n \text{ is odd},\\
	0, & \text{ if } n \text{ is even}.
	\end{cases}
\]
\end{knownlem}
\begin{proof}
For even $n$, this is easy to see: the assumption $a \equiv -1 \bmod{p}$  implies $a^n +1 \equiv 2 \not\equiv 0 \pmod{p}$. For odd $n$, the statement follows from Lemma~\ref{lem:vp-1} noting that $v(a^n + 1 ) = v_p((-a)^n - 1)$.
\end{proof}

\begin{knownlem}[{\cite[Corollary 2.1.23]{Cohen2007I}}]\label{lem:v2_powers-1}
Let $a \in \Z $ be odd. Then for any positive integer $n$ we have 
\[
	v_2(a^n-1)= \begin{cases}
	v_2(a-1) & \text{if } n \text{ is odd}, \\
	v_2(a^2 - 1) + v_2(n) - 1 & \text{if } n \text{ is even}.
	\end{cases}
\]
\end{knownlem}

The next lemma is formulated only for odd $n$ and is easy to check.

\begin{knownlem}\label{lem:v2_oddexpo}
Let $a$ and $n$ be positive integers  with $n$ odd. Then 
\[
	v_2(a^n+1)=v_2(a+1).
\]
\end{knownlem}

\begin{knownlem}\label{lem:v3} For any positive integer $n$, we have
\[
	v_3(2^n - 1) = \begin{cases}
	0, &\text{if }n \text{ is odd},\\
	v_3(n) + 1, &\text{if } n \text{ is even},
	\end{cases}
	\qquad \text{and} \qquad
	v_3(2^n + 1) = \begin{cases}
	v_3(n) + 1, &\text{if } n \text{ is odd},\\
	0, &\text{if } n \text{ is even}.
	\end{cases}
\]
\end{knownlem}
\begin{proof}
The formula for $v_3(2^n + 1)$ follows immediately from Lemma~\ref{lem:vp+1}. 

It is clear that $v_3(2^n - 1)= 0 $ for odd $n$, because $2^n \equiv 2 \pmod{3}$ for $n$ odd. If $n$ is even, then note that $v_3(2^n - 1)= v_3(4^{n/2} - 1)$ and appeal to Lemma~\ref{lem:vp-1}.
\end{proof}

\begin{knownlem}[Zsigmondy's theorem \cite{Zsigmondy1892}]\label{lem:Zsigmondy}
Let $a > b \geq 1$ be coprime integers. Then for any integer $n\geq 1$, there is a prime number $p$ (called a primitive divisor) that divides $a^n - b^n$ and does not divide $a^k - b^k$ for any $1\leq k < n$, with the following exceptions:
\begin{itemize}
\item $n = 1$ and $a-b = 1$,
\item $n = 2$ and $a+b$ is a power of 2,
\item $n = 6$ and $a=2, b=1$.
\end{itemize}
Similarly, $a^n + b^n$ has at least one primitive divisor, with the exception
\begin{itemize}
\item $n = 3$ and $a=2, b=1$.
\end{itemize}
\end{knownlem}

\subsection{Lower bounds for linear forms in logarithms}

We will have use of  lower bounds for linear forms in more than two logarithms only in the integer case, and appeal to work of  Matveev.

\begin{knownthm}[{Matveev, 2000, \cite[Corollary 2.3]{Matveev2000}}]\label{thm:Matveev}
Let $a_1, \ldots ,  a_n \in \Z_{\geq 2}$, $b_1, \ldots ,  b_n \in \Z$ and
\[
	\Lambda = b_1 \log a_1 + \dots + b_n \log a_n.
\]
Suppose that $\Lambda\neq 0$. Then we have
\[
	\log |\Lambda|
	> - 1.4 \, n^{4.5} \, 30^{n+3} \log a_1 \cdots \log a_n (1+ \log \max\{|b_1|, \ldots ,  |b_n|\}).
\]
\end{knownthm}

The next lemma describes the way in which we will apply the LLL-algorithm to reduce the lower bounds for linear forms in four logarithms. It is a straightforward variant of \cite[Lemma VI.1]{Smart1998}.

\begin{knownlem}[LLL reduction]\label{lem:LLL}
Let $\gamma_1,\gamma_2,\gamma_3, \gamma_4$ be positive real numbers and $x_1,x_2,x_3,x_4$ integers and let
\[
	|\Lambda|
	:= |x_1 \log \gamma_1 + x_2 \log \gamma_2 + x_3 \log \gamma_3 + x_4 \log \gamma_4|
	\neq 0.
\]
Assume that the $x_i$ are bounded in absolute values by some constant $M$ and choose a constant $C>M^4$.
Consider the matrix
\[
	A= \begin{pmatrix}
	1 & 0 & 0 & 0 \\
	0 & 1 & 0 & 0 \\
	0 & 0 & 1 & 0 \\
	[C \log \gamma_1] & [C \log \gamma_2] &[C \log \gamma_3] &[C \log \gamma_4]
	\end{pmatrix},
\]
where $[x]$ denotes the nearest integer to $x$. The columns of $A$ form a basis of a lattice. Let $B$ be the matrix that corresponds to the LLL-reduced basis and let $B^*$ be the matrix corresponding to the Gram-Schmidt basis constructed from $B$. Let $c$ be the Euclidean norm of the smallest column vector of $B^*$ and set $S:=3M^2$ and $T:=(1+4M)/2$.
If $c^2 > T^2 + S$, then 
\[
	|\Lambda|
	> \frac{1}{C}\left( \sqrt{c^2-S} - T \right).
\]
\end{knownlem}

For any algebraic number $\alpha$ of degree $d$ over $\Q$, the logarithmic height is defined as usual by the formula
\[
	\h(\alpha) 
	= \frac{1}{d} \left( \log |a| + \sum_{i=1}^d \log \max \{1, |\alpha^{(i)}| \} \right),
\]
where $a$ is the leading coefficient of the minimal polynomial of $\alpha$ over $\Z$, and
the $\alpha^{(i)}$'s are the conjugates of $\alpha$ in the field of complex numbers.

For the bounds on the equation $x^2 - 2 = y^n$ in Section~\ref{sec:x2-2}, we will require the full strength of Laurent's theorem for linear forms in two logarithms.

\begin{knownthm}[{Laurent, 2008, \cite[Theorem 2]{Laurent2008}}]\label{thm:Laurent}
Let $a_1, a_2, h, \rho$ and $\mu$ be real numbers with $\rho >1$ and $1/3 \leq \mu < 1$. Set
\begin{align*}
	&\sigma  = \frac{1 + 2\mu - \mu^2}{2}, \quad
	\lambda = \sigma \log \rho, \quad
	H = \frac{h}{\lambda} + \frac{1}{\sigma}, \\
	& \omega = 2 \left( 1 + \sqrt{1 + \frac{1}{4 H^2}} \right), \quad 
	\theta = \sqrt{1 + \frac{1}{4 H^2}} + \frac{1}{2H}.
\end{align*}
Consider the linear form 
\[
	\Lambda = b_2 \log \alpha_2 - b_1 \log \alpha_1,
\]
where $b_1$ and $b_2$ are positive integers, and $\alpha_1$ and $\alpha_2$ are multiplicatively independent real algebraic numbers $>1$.
Put $D = [\Q(\alpha_1, \alpha_2): \Q]$, and assume that
\begin{align*}
&\text{(1)} \quad
h \geq \max \left\{
	D \left( 
		\log \left( 
			\frac{b_1}{a_2} + \frac{b_2}{a_1} 
			\right)
		+ \log \lambda + 1.75 \right) + 0.06,
	\lambda,
	\frac{D \log 2}{2}
	\right\},\\
&\text{(2)} \quad
a_i \geq \max\{ 1, (\rho-1) \log \alpha_i + 2 D \h(\alpha_i)\} \quad (i = 1,2),\\
&\text{(3)} \quad
a_1 a_2 \geq \lambda^2.
\end{align*}
Then
\[
\log |\Lambda| 
\geq - C \left( h + \frac{\lambda}{\sigma} \right) ^2 a_1 a_2
- \sqrt{\omega \theta} \left( h + \frac{\lambda}{\sigma} \right)
- \log \left( 
	C'  \left( h + \frac{\lambda}{\sigma} \right)^2 a_1 a_2
	\right)
\]
with
$$
C = \frac{\mu}{\lambda^3 \sigma} \left(
		\frac{\omega}{6} + \frac{1}{2} \sqrt{
			\frac{\omega^2}{9} 
			+ \frac{8 \lambda \omega^{5/4} \theta^{1/4}}{3 \sqrt{a_1 a_2} H^{1/2}}
			+ \frac{4}{3} \left( \frac{1}{a_1} + \frac{1}{a_2} \right)
			\frac{\lambda \omega}{H}		
		}
	\right)^2 \; \; \mbox{ and } \; \; 
C' = \sqrt{\frac{C \sigma \omega \theta}{\lambda^3 \mu}}.
$$
\end{knownthm}

We will also have use of a somewhat weaker but easy to apply version of this, a special case of Corollary 2 of Laurent \cite{Laurent2008}:

\begin{cor}\label{cor:Laurent}
Let $\alpha_1$ and $\alpha_2$ be multiplicatively independent real algebraic numbers $>1$ and $b_1,b_2 \in \Z$ not both zero.
Set $D = [\Q(\alpha_1, \alpha_2): \Q]$.
Then for $(C,m) = (20.3, 18)$ or $(C,m) = (17.9, 30)$ we have
\[
	\log | b_2 \log \alpha_2 - b_1 \log \alpha_1|
	\geq -C \cdot D^4 \left( \max \{ \log b^\prime + 0.38, m/D, 1 \} \right)^2 \log A_1 \log A_2,
\]
where $\log A_i \geq \max \{ h(\alpha_i), |\log \alpha_i|/D, 1/D \}$ and
\[
	b^\prime = \frac{|b_1|}{D \log A_2} + \frac{|b_2|}{D \log A_1}.
\]
\end{cor}

\section{Solution of equation \eqref{eq:myeqs}}\label{sec:myeqs}

\begin{newprop}\label{prop:sit1}
The only solutions to the equation
\begin{equation}\label{eq:prop_sit1}
	2^y ( 2^x + 1)^z - 1
	= 3^{r} (2^{x+1} + 1)^{w}
\end{equation}
in integers 
$x \geq 3$, $y \geq 1$, $z \geq 0$, $w \geq 0$, and $r$ arbitrary are given by 
\[
(x,y,z,w,r) \in \{(x,1,0,0,0), (x,1,1,1,0),  (x,2,0,0,1)\}.
\]
\end{newprop}

\begin{newprop}\label{prop:sit2}
The only solutions to the equation
\begin{equation}\label{eq:prop_sit2}
	2^y ( 2^x - 1)^z + 1
	= 3^{r} (2^{x+1} - 1)^{w}
\end{equation}
in integers 
$x \geq 3$, $y \geq 1$, $z \geq 0$, $w \geq 0$, and $r$ arbitrary are given by 
\[
(x,y,z,w,r) \in \{(x, x+2, 1, 2, 0), (x,1,1,1,0), (x,1,0,0,1), (x,3,0,0,2), (3,2,0,1,-1)\}.
\]
\end{newprop}

We will first focus our attention on some special cases ($z=0$ and $r=0$) of Propositions~\ref{prop:sit1} and \ref{prop:sit2}, which correspond to the actual  solutions. Then we will provide a joint proof for both propositions that there exist no further solutions.

We start with $z=0$ in Proposition \ref{prop:sit1}.

\begin{newlem}\label{lem:sit1_z0}
The only solutions to the equation
\begin{equation}\label{eq:prop_sit1_z0}
	2^y - 1
	= 3^{r} (2^{x+1} + 1)^{w}
\end{equation}
in integers 
$x \geq 3$, $y \geq 1$, $w \geq 0$, and $r$ arbitrary are given by 
\[
(x,y,w,r) \in \{(x,1,0,0),  (x,2,0,1)\}.
\]
\end{newlem}
\begin{proof}
First, assume  that $w = 0$. Then \eqref{eq:prop_sit1_z0} becomes $2^y - 1 = 3^r$, and  Theorem \ref{thm:Catalan} implies that either $r=0$ and $y = 1$, or $r = 1$ and $y = 2$. This gives us the two solutions $(x,y,w,r) = (x,1,0,0)$ and $(x,y,w,r) = (x,2,0,1)$. 

Now assume that $w\geq 1$.
Since $x\geq 3$, we have that $(2^{x+1}+1)^w \equiv 1 \pmod{8}$. Moreover, $3^r \equiv 1,3 \pmod{8}$ (this holds also if $r$ is negative: $1^{-1} = 1$ and $3^{-1} = 3$ in the ring of integers modulo 8). Thus, equation \eqref{eq:prop_sit1_z0} implies $2^y - 1 \equiv 1,3 \pmod{8}$, so $y = 1$ or $y = 2$. 
Then the left hand side of \eqref{eq:prop_sit1_z0} is $2^y -1 \in \{1,3\}$, which implies that on the right hand side $(2^{x+1}+1)^w$ is a power of 3. This is impossible for $w\geq 1$ and $x\geq 3$, again from Theorem \ref{thm:Catalan}.
\end{proof}

Next, we consider the case $r = 0$ in Proposition~\ref{prop:sit1}.

\begin{newlem}\label{lem:sit1_r0}
The only solutions to the equation
\begin{equation}\label{eq:prop_sit1_r0}
	2^y ( 2^x + 1)^z - 1
	= (2^{x+1} + 1)^{w}
\end{equation}
in integers 
$x \geq 3$, $y \geq 1$, $z \geq 0$, and $w \geq 0$ are given by 
\[
(x,y,z,w) \in \{(x,1,0,0), (x,1,1,1)\}.
\]
\end{newlem}
\begin{proof}
Considering equation \eqref{eq:prop_sit1_r0} modulo 4, we immediately find that $y = 1$. Therefore, we have the equation
\begin{equation}\label{eq:prop_sit1_r0_2}
	2 (2^x + 1)^{z}
	= (2^{x+1} + 1)^w + 1,
\end{equation}
and recall that $x \geq 3$, $z \geq 0$ and $w \geq 0$.

First, note that if $z = 0$ or $w = 0$, or if  $z = 1$ or $w = 1$, we immediately obtain the solutions $(x,y,z,w) = (x,1,0,0)$
and $(x,y,z,w) = (x,1,1,1)$, respectively. From now on, assume that $z\geq 2$ and $w\geq 2$.

Let $p>2$ be a prime factor of $2^x + 1$.  From the left hand side of \eqref{eq:prop_sit1_r0_2} we have
\[
	v_p\left(2 (2^x + 1)^{z}\right)
	= z \cdot v_p(2^x + 1).
\]
Considering  the right hand side, note that $(2^{x+1} + 1) + 1 = 2\cdot (2^x + 1)$ is divisible by $p$, so from Lemma~\ref{lem:vp+1},
\begin{align*}
	v_p \left( (2^{x+1} + 1)^w + 1 \right)
	&= v_p((2^{x+1} + 1) + 1) + v_p(w)
	= v_p(2^{x} + 1) + v_p(w).
\end{align*}
Comparing the valuations of both sides of \eqref{eq:prop_sit1_r0_2}, we obtain
\[
	z \cdot  v_p(2^x + 1)
	= v_p(2^x + 1) + v_p(w),
\]
and so 
\[
	z - 1
	\leq (z-1) \cdot v_p(2^x+1)
	= v_p(w)
	\leq \log_p(w).
\]
Roughly speaking, this means that $z$ is much smaller than $w$, contradicting equation \eqref{eq:prop_sit1_r0_2}. We now quantify this argument.
We estimate both sides of \eqref{eq:prop_sit1_r0_2}. Noting that for $x\geq 3$ we have $2^x+1 \leq 2^{x+0.2}$, we obtain
\[
	2^{(x+1)w}
	\leq (2^{x+1} + 1)^w + 1
	= 2 (2^x + 1)^{z}
	\leq 2(2^{x+ 0.2})^{z}
	= 2^{1 + (x + 0.2)z}.
\]
Thus, using $z \leq v_p(w) +1$, we obtain
\[
	(x+1)w
	\leq 1 + (x + 0.2)z
	\leq 1 + (x+0.2)(v_p(w) +1).
\]
Since we are assuming $w\geq 2$, we have in particular $1 \leq 0.8 w$. Thus, the above inequality implies
\[
	(x+0.2) w \leq (x+0.2)(v_p(w) +1),
\]
and so $w\leq v_p(w) + 1$, which is impossible for any $p>2$ and $w\geq 2$.
\end{proof}

We continue with the case $z=0$ in Proposition \ref{prop:sit2}.

\begin{newlem}\label{lem:sit2_z0}
The only solutions to the equation
\begin{equation}\label{eq:prop_sit2_z0}
	2^y + 1
	= 3^{r} (2^{x+1} - 1)^{w}
\end{equation}
in integers 
$x \geq 3$, $y \geq 1$, $w \geq 0$, and $r$ arbitrary are given by 
\[
(x,y,w,r) \in \{(3,2,1,-1), (x,1,0,1), (x,3,0,2)\}.
\]
\end{newlem}
\begin{proof}
Assume for a moment that $w = 0$. Then \eqref{eq:prop_sit2_z0} becomes $2^y + 1 = 3^r$, and Theorem~\ref{thm:Catalan} implies  that either $r=1$ and $y = 1$, or $r = 2$ and $y = 3$. 
This gives us the two solutions $(x,y,w,r) = (x,1,0,1)$ and $(x,y,w,r) = (x,3,0,2)$. 
Now assume that $w\geq 1$.

If $y = 1$, then $2^y + 1 = 3$, and so $2^{x+1} - 1$ has to be a power of 3, contradicting Theorem~\ref{thm:Catalan} for $x \geq 2$.
If $y = 2$, then $2^y + 1 = 5$, so $w=1$ and $2^{x+1} - 1 = 3^{|r|}\cdot 5$. Since $2^4 - 1 = 15 = 3\cdot 5$, by Lemma \ref{lem:Zsigmondy}, we find that $x+1\leq 4$. Indeed, $x=3$ gives us the solution $(x,y,w,r) = (3,2,1,-1)$.

Finally, assume that $y\geq 3$. We consider equation \eqref{eq:prop_sit2_z0} modulo 8:
\[
	1 \equiv 3^r (-1)^w \pmod{8}.
\]
This implies that $r$ and $w$ are both even.
But then the right hand side of \eqref{eq:prop_sit2_z0} is a square, contrary to Theorem \ref{thm:Catalan}.
\end{proof}

Finally, we solve the case $r=0$ in Proposition~\ref{prop:sit2}.

\begin{newlem}\label{lem:sit2_r0}
The only solutions to the equation
\begin{equation}\label{eq:prop_sit2_r0}
	2^y ( 2^x - 1)^z + 1
	= (2^{x+1} - 1)^{w},
\end{equation}
in integers 
$x \geq 3$, $y \geq 1$, $z \geq 0$, and $w \geq 0$ are given by 
\[
	(x,y,z,w) \in \{(x, x+2, 1, 2), (x,1,1,1)\}.
\]
\end{newlem}
\begin{proof}
By Lemma~\ref{lem:sit2_z0}, we may assume $z\geq 1$.
We rewrite equation \eqref{eq:prop_sit2_r0} as
\begin{equation}\label{eq:prop_sit2_r0_proof}
2^y ( 2^x - 1)^z
	= (2^{x+1} - 1)^{w} -1.
\end{equation}
Clearly $w= 0$ is impossible.
Considering the right hand side of \eqref{eq:prop_sit2_r0_proof},  Lemma~\ref{lem:v2_powers-1} implies that
\begin{align*}
	 v_2 ((2^{x+1} - 1)^{w} - 1)
	\leq v_2 ((2^{x+1} - 1)^{2} - 1 ) + v_2(w) - 1
	= x + 1 + v_2 (w),
\end{align*}
whence
\begin{equation}\label{eq:ybound}
	y \leq x + 1 + v_2 (w).
\end{equation}

Next, let $p>2$ be a prime factor of $2^x-1$. Then the $p$-adic valuation of the left hand side of \eqref{eq:prop_sit2_r0_proof} is $z \cdot v_p(2^x - 1)$. For the right hand side, note that $2^{x+1} - 1 = 2(2^x -1) + 1 \equiv 1 \pmod{p}$, so we get from Lemma~\ref{lem:vp-1} that
\begin{align*}
	v_p((2^{x+1} - 1)^{w} - 1)
	&= v_p((2^{x+1} - 1) - 1) + v_p(w)\\
	&= v_p(2 \cdot (2^x - 1)) + v_p(w)
	= v_p(2^x - 1) + v_p(w).
\end{align*}
Comparing both sides of \eqref{eq:prop_sit2_r0_proof}, we conclude that 
 \[
	z \cdot v_p(2^x - 1) 
	= v_p(2^x - 1) + v_p(w),
\]
and so
\begin{equation}\label{eq:zbound}
	z-1 
	\leq (z-1) \cdot  v_p(2^x -1)
	= v_p(w).
\end{equation}
Since $z\geq 1$ and $x\geq 3$, we find from \eqref{eq:prop_sit2_r0} that
\begin{align*}
	2^{(x+0.9)w}
	\leq (2^{x+1} - 1)^{w}
	= 2^y ( 2^x - 1)^z + 1
	\leq 2^{y + xz}.
\end{align*}
Thus,
\[
	(x+0.9) \cdot w
	\leq y + xz,
\]
and combining this with \eqref{eq:ybound} and \eqref{eq:zbound}, 

\begin{align*}
	(x+0.9) \cdot w
	&\leq (x + 1 + v_2(w)) + x \cdot (v_p(w) + 1)\\
	&\leq (x+ 0.9) \cdot \left(v_p(w) + 2 + \frac{1+v_2(w)}{3.9}\right),
\end{align*}
and so
\begin{equation*}
	w 
	\leq v_p(w) + 2 + \frac{1+v_2(w)}{3.9}
	\leq \log_3(w) + 2 + \frac{1+\log_2(w)}{3.9}.
\end{equation*}
We conclude that $w \leq 4$. 

For $w\leq 4$, we have that ${1+\log_2(w)}< {3.9}$, so the above inequality actually implies 
\begin{equation}\label{eq:wdone}
	w \leq v_p(w) + 2,
\end{equation} 
where $p$ is any prime dividing $2^x - 1$. 

If $w = 4$, we immediately get a contradiction from \eqref{eq:wdone}. 

If $w = 3$, inequality \eqref{eq:wdone} leads to a contradiction for any $p>3$, and it is clear that $2^x-1$ always has a prime factor larger than 3 for $x\geq 3$.

If $w = 2$, we get $ y = v_2((2^{x+1} - 1)^{2} - 1) = x+2$, and \eqref{eq:prop_sit2_r0_proof} becomes
\[
	2^{x+2}(2^x -1)^z 
	= (2^{x+1} - 1)^{2} -1
	= 2^{x+2}(2^{x} - 1).
\]
Thus, $z = 1$, and we have the solution $(x,y,z,w) = (x, x+2, 1, 2)$.

Finally, assume that $w = 1$. Then,
considering \eqref{eq:prop_sit2_r0_proof} modulo 4, we see that $y = 1$, and the equation  becomes $2 \cdot (2^x -1)^z = (2^{x+1} - 1)^{1} -1 = 2^{x+1} - 2$. This leads to the solution $(x,y,z,w) = (x,1,1,1)$. 
\end{proof}

Now, in order to prove Propositions \ref{prop:sit1} and \ref{prop:sit2},
we need to show that there are no solutions to equations \eqref{eq:prop_sit1} and \eqref{eq:prop_sit2} in the remaining cases.

\begin{proof} [Proof of Propositions \ref{prop:sit1} and \ref{prop:sit2}]
We need to solve the two equations
\begin{equation}\label{eq:prop_proof}
	2^y ( 2^x \pm 1)^z \mp 1
	= 3^{r} (2^{x+1} \pm 1)^{w},
\end{equation}
in integers 
$x \geq 3$, $y \geq 1$, $z \geq 0$, $w \geq 0$, and $r$ arbitrary, and in view of Lemmata \ref{lem:sit1_z0}, \ref{lem:sit1_r0}, \ref{lem:sit2_z0} and \ref{lem:sit2_r0} we may assume $z\geq 1$ and $r\neq 0$.
 
We set
\[
	M := \max\{y,z,w,10^{20}\}.
\]

First, we bound $|r|$ in terms of $x$ and $M$.
If $r$ is positive, then we have
\[
	3^r
	\leq 
	3^{r} (2^{x+1} \pm 1)^{w}
	= 2^y ( 2^x \pm 1)^z \pm 1
	< 2^{y + (x + 0.2)z + 0.2},
\]
and so 
\begin{align*}
	r 
	&< (y + (x + 0.2)z + 0.2) \cdot \frac{\log 2}{\log 3} \\
	&< (M + (x + 0.2)M + 0.2)\cdot \frac{\log 2}{\log 3}
	< M x.
\end{align*}
If $r$ is negative, then since all numbers in \eqref{eq:prop_proof} are integers, we must have
\begin{align*}
	|r|
	&\leq v_3((2^{x+1} \pm 1)^{w})
	= w \cdot  v_3(2^{x+1} \pm 1)\\
	&\leq w \cdot (v_3(x+1) + 1)
	< w \cdot (\log (x+1) + 1)
	< M x,
\end{align*}
where we used Lemma~\ref{lem:v3}.
Overall, we have
\begin{equation}\label{eq:prop_rbound1}
|r|
< M x.
\end{equation}

Now we construct a linear form in logarithms from the main equation \eqref{eq:prop_proof}: Subtracting $2^y(2^x \pm 1)^z$ and taking absolute values, we get
\[
	 1
	= | 3^{r} (2^{x+1} \pm 1)^{w} - 2^y ( 2^x \pm 1)^z |.
\]
Dividing by $2^y(2^x \pm 1)^z$ and using $|\log \xi |< 2 |\xi -1 |$ for $|\xi -1| <0.5$, we obtain
\begin{multline}\label{eq:prop_linform}
	|\Lambda|
	:=\left|
		r \log 3 + w \log (2^{x+1} \pm 1)
		- y \log 2 - z \log (2^x \pm 1)
	\right|\\
	< \frac{2}{2^y ( 2^x \pm 1)^z}
	< \frac{1}{2^{y + (x-0.2) z -1}}.
\end{multline}
For later, note that by construction $\Lambda \neq 0$.

We do not have any bounds for $\log(2^{x+1} \pm 1)$ and $\log(2^{x} \pm 1)$ yet. To obtain these, we will construct another linear form in logarithms. We use the $L$-notation, where by $L(a)$ we mean a number which has absolute value at most $a$. Since $1/(2^b - 1) < 2^{-b + 0.2}$ for $b\geq 3$,
\begin{align*}
	\log (2^{x+1} \pm 1) &= \log (2^{x+1}) + L(2^{-(x+1) + 0.2}) = (x+1) \log 2 + L(2^{-x - 0.8}),\\
\log (2^x \pm 1) &= \log (2^x) + L(2^{-x + 0.2}) = x \log 2 + L(2^{-x + 0.2}).	
\end{align*}
Thus, we obtain from \eqref{eq:prop_linform} that
\begin{multline}\label{eq:prop_linform_simpl}
		|\Lambda'| :=
		\left|
		r \log 3 
		+ (w(x+1) - y - zx)\log 2 
	\right|\\
	< \frac{1}{2^{y + (x-0.2) z -1}} + \frac{w}{2^{x+ 0.8}} + \frac{z}{2^{x - 0.2}} 
	< \frac{2M}{2^{x-0.2}}.
\end{multline}
Note that $\Lambda' \neq 0$ because $r\neq 0$ and 2 and 3 are multiplicatively independent.

If $2^{x-0.2} < 4M$, then $x < (\log 4 + \log M + 0.2) / \log 2$, and we can immediately skip to \eqref{eq:x2_sit2}. Assume that $2^{x-0.2} \geq 4M$. Then $|\Lambda'| < 1/2$. In view of \eqref{eq:prop_rbound1}, both coefficients in $|\Lambda'|$ are  bounded by $(x+1)M < 1.34 \cdot x \cdot M$.

We apply Lemma \ref{cor:Laurent} to $|\Lambda'|$ with $a_1 = 3$, $b_1 = r$, $a_2 = 2$, $b_2 = w(x+1) - y - zx$:
\begin{align*}
	\log A_1 & = \log 3 = \max\{\log 3, 1\},\\
	\log A_2 & = 1 =  \max\{\log 2, 1\},\\
	\log b'
		& = \log\left(\frac{|r|}{1} + \frac{|w(x+1) - y - zx|}{\log 3} \right)
		< \log (2 \cdot x \cdot M ).
\end{align*}
Note that
$\log (2 \cdot x \cdot M ) \geq \log (2 \cdot 3 \cdot 10^{20})> 47$,
 so Lemma \ref{cor:Laurent} gives us
\[
	\log |\Lambda'|
	\geq - 17.9 \cdot (\log (2 \cdot x \cdot M ) + 0.38)^2 \cdot \log 3 \cdot 1.
\]
Comparing this with the bound from \eqref{eq:prop_linform_simpl}, we obtain
\[
	(x - 0.2) \log 2 < \log (2M) + 17.9 \cdot (\log (2 \cdot x \cdot M ) + 0.38)^2 \cdot \log 3.
\]
This inequality implies a bound of the shape $x \ll (\log M)^2$; we now compute the constant. Collecting some of the terms and constants, we get
\[
	x \cdot 0.93 \cdot \log 2 < 20.6 \cdot ( \log x + \log M)^2,
\]
and dividing by $0.93 \cdot \log 2$
\begin{equation}\label{eq:x_sit2}
	x < 32 \cdot ( \log M + \log x)^2.
\end{equation}
Taking logarithms and using $\log M \geq \log (10^{20})>46$, we get
\begin{align*}
	\log x 
	&< 3.47 + 2 \cdot (\log \log M + \frac{\log x}{46})
	< 3.47 + 2 \cdot \log \log M + 0.05 \log x.
\end{align*}
Subtracting $0.05 \log x$, and dividing by $0.95$, we obtain
\[
	\log x < 3.66 + 2.11 \log \log M.
\]
Plugging into \eqref{eq:x_sit2} yields
\begin{align}\label{eq:x2_sit2}
	x &< 32  \cdot (\log M + 3.66 + 2.11 \log \log M)^2
	\nonumber \\
	&< 32 \cdot (1.26 \log M)^2
	< 50.9 \cdot (\log M)^2.
\end{align}

Now that we have a bound for $x$ in terms of $\log M$, we go back to \eqref{eq:prop_linform} and apply \hyperref[thm:Matveev]{Matveev's theorem} to $|\Lambda|$.
Note that by \eqref{eq:prop_rbound1} $r$ is now bounded by
\[
	|r| 
	< 50.9 \cdot M \cdot (\log M)^2.
\]
The other coefficients in $\Lambda$ are bounded by $M$.
Thus, we obtain
\begin{align*}
	\log |\Lambda|
	&> - C_{\text{Matv},4} \cdot ( 1+ \log (50.9 \cdot M \cdot (\log M)^2)) \cdot \log 3 \cdot \log(2^{x+1} + 1) \cdot \log 2 \cdot \log (2^x+1)\\
	& > - 8.6 \cdot 10^{12} \cdot ( 1+ \log (50.9 \cdot M \cdot (\log M)^2)) \cdot x^2\\
	& > - 2.23 \cdot 10^{16} \cdot ( 1+ \log (50.9 \cdot M \cdot (\log M)^2)) (\log M)^4 ,
\end{align*}
where we used \eqref{eq:x2_sit2} for the last estimation.
Combining this with the bound from \eqref{eq:prop_linform},
\[
	\log |\Lambda|
	< - (y + (x-0.2) z -1) \log 2,
\]
we obtain, upon division by $\log 2$,
\begin{equation}\label{eq:prop_yzx}
	y + (x-0.2) z -1
	< 3.22 \cdot 10^{16} \cdot ( 1+ \log (50.9 \cdot M \cdot (\log M)^2)) (\log M)^4.
\end{equation}
In particular, this bounds $y$ and $z$. 
In order to get a bound for $w$ as well, we return to equation \eqref{eq:prop_proof}:
\begin{equation*}\label{eq:prop_rep}
	2^y ( 2^x \pm 1)^z \mp 1
	= 3^{r} (2^{x+1} \pm 1)^{w}.
\end{equation*}
Since $(2^{x+1} \pm 1)$ cannot be a power of 3, $(2^{x+1} \pm 1)$ is divisible by at least one prime factor $p\geq 5$, whence  the right hand side of the above equation is at least $5^w$. 
But then we have
\[
	5^w
	\leq 3^{r} (2^{x+1} \pm 1)^{w}
	= 2^y ( 2^x \pm 1)^z \mp 1
	\leq 2^{y + (x + 0.2)  z},
\]
which implies
\begin{equation}\label{eq:wbound}
	w
	< (y + (x + 0.2)  z) \cdot \frac{\log 2}{\log 5}
	< y + (x-0.2) z -1,
\end{equation}
so $w$ is bounded by the bound in \eqref{eq:prop_yzx} as well.
Overall, we have
\[
	 M
	 = \max\{y,z,w,10^{20}\}
	 < 3.22 \cdot 10^{16} \cdot ( 1+ \log (50.9 \cdot M \cdot (\log M)^2)) (\log M)^4.
\]
Solving this inequality, we have that
\begin{equation}\label{eq:prop_Mbound}
	M < 2.72 \cdot 10 ^{25}.
\end{equation}
Then we obtain from \eqref{eq:x2_sit2} that
\begin{equation}\label{eq:prop_xbound}
	x < 50.9 \cdot (\log M)^2
	< 1.75 \cdot 10^5.
\end{equation}

Next, we reduce the bound for $x$ by returning  to \eqref{eq:prop_linform_simpl} and computing the continued fraction of $\log 3 / \log 2$. 
From \eqref{eq:prop_linform_simpl} we have
\[
	\left| 
		\frac{\log 3}{\log 2} - \frac{w(x+1) - y - zx}{-r}
	\right|
	< \frac{2M}{2^{x-0.2} \log 2}.
\] 
By \eqref{eq:prop_rbound1}, \eqref{eq:prop_Mbound} and \eqref{eq:prop_xbound}, the denominator in the above approximation is bounded by
\[
	|r|
	<  M x
	< 4.76 \cdot 10^{30}.
\]
Then, computing the smallest convergent to $\log 3/ \log 2$ with denominator at least $4.76 \cdot 10^{30}$, we get that
\[
	3.43 \cdot 10^{-64} 	
	<\left| 
		\frac{\log 3}{\log 2} - \frac{p_{61}}{q_{61}}
	\right|	
	< \left| 
		\frac{\log 3}{\log 2} - \frac{w(x+1) - y - zx}{-r}
	\right|
	< \frac{2M}{2^{x-0.2} \log 2},
\] 
and so
\[
	x < \frac{- \log(3.43 \cdot 10^{-64}) + \log (2 \cdot 2.72 \cdot 10 ^{25}) - \log \log 2}{\log 2}
	< 297.
\]

Finally, we need to solve the two equations \eqref{eq:prop_proof} for each $3 \leq x \leq 296$:
\begin{align}
2^y ( 2^x + 1)^z - 1 &= 3^{r} (2^{x+1} + 1)^{w}, \label{eq:plus}\\
2^y ( 2^x - 1)^z + 1 &= 3^{r} (2^{x+1} - 1)^{w}. \label{eq:minus}
\end{align}
First, note that we have restrictions on the parity of $x$: In the case of \eqref{eq:plus}, assume for a moment that $x$ is odd. Then $3\mid 2^x+1$, so the left hand side of \eqref{eq:plus} is not divisible by 3. Moreover, $2^{x+1} + 1$ on the right hand side of \eqref{eq:plus} is not divisible by 3. But then  $r = 0$, which has been dealt with in Lemma~\ref{lem:sit1_r0}. Therefore, it suffices to consider even $x$. By an analogous argument, in the case of \eqref{eq:minus} it suffices to consider odd $x$.

Next, recall that by \eqref{eq:prop_rbound1} and \eqref{eq:prop_Mbound} we have
\[
	|r|
	< M x
	<  2.73 \cdot 10 ^{25}  \cdot 296
	< 8.06 \cdot 10^{27}.
\]
Now for each $3 \leq x \leq 296$ (with the parity restrictions), we do the following. We apply Lemma \ref{lem:LLL}
to find a good lower bound for $|\Lambda|$ (all coefficients are bounded by $8.06 \cdot 10^{27}$). Let us denote the lower bound by $\ell(x)$, i.e.\ the LLL algorithm gives us
\[
	|\Lambda| > \ell(x)
\]
for every $3 \leq x \leq 296$.
Thus, from \eqref{eq:prop_linform} we get
\[
(y + (x-0.2) z -1) \log 2 
	< - \log (\ell(x)),
\]
and so 
\[
	y + (x-0.2) z -1
	< - \frac{\log \ell(x)}{\log 2}
	=: U(x).
\]
Then for each $x$, we only need to check all possible solutions with $y + (x-0.2) z -1 < U(x)$ and $y, z\geq 1$. 
For each such triple $(x,y,z)$ we do the following. We compute the left hand side of \eqref{eq:plus} or \eqref{eq:minus} and we cancel out all powers of 3, i.e.\ we compute
\[
	(2^y ( 2^x \pm 1)^z \mp 1) \cdot 3^{- v_3(2^y ( 2^x \pm 1)^z \mp 1)}
	=: q(x,y,z).
\]
Note that this is an easy computation even for very large numbers where we do not know the prime factorization. 
Then equation \eqref{eq:plus} or \eqref{eq:minus} gives us
\[
	q(x,y,z) = \left( (2^{x+1} \pm 1) \cdot 3^{- v_3(2^{x+1} \pm 1) } \right) ^w,
\]
and we can easily check whether $q(x,y,z)$ is a power of $(2^{x+1} \pm 1)\cdot 3^{- v_3(2^{x+1} \pm 1)}$.

A quick computation using SageMath \cite{sagemath} reveals that in all cases we get a bound $U(x) < 449$, and that there are no solutions except for those with $y = z = 1$ and (in the case of \eqref{eq:minus}) $y=x+2$, $z = 1$. If $y = z = 1$, we have $2^1 \cdot (2^x \pm 1) ^1 \mp 1 = 2^{x+1} \pm 1$, and if $y=x+2$, $z = 1$, we have $2^{x+2} \cdot (2^x - 1) ^1 + 1 = (2^{x+1} - 1)^2$. Thus, in all the exceptional solutions that show up we have $r=0$, which we have dealt with in Lemmata \ref{lem:sit1_r0} and \ref{lem:sit2_r0} and excluded at the beginning of the proof of Proposition \ref{prop:sit1} and \ref{prop:sit2}.
\end{proof}

\begin{rem}[on the LLL reduction in the proof of Propositions \ref{prop:sit1} and \ref{prop:sit2}]
As $x$ grows, the logarithms $\log(2^x \pm 1)$, $\log(2^{x+1} \pm 1)$ get very close to multiples of $\log 2$. Therefore, the constant $C$ needs to be larger than the usual $M^4$. Usually, one starts with $C \approx M^4$ and then increases $C$ (for example by a factor of 10) until the reduction works. In our computations, for each $x$ we start with the $C$ from the previous $x-1$. This may make $C$ a bit larger than necessary, but it saves time and still leads to very good bounds.

The Sage code for the LLL reduction is available at: \url{https://cocalc.com/IngridVukusic/ConsecutiveTriples/LLLreduction}
\end{rem}

\section{Proof of Theorem~\ref{thm:a2}}\label{sec:a2}

In this section, we prove Theorem~\ref{thm:a2}, i.e., we prove that if we have integers $2<b<c$ such that $(2,b,c)$, $(3,b+1,c+1)$ and $(4,b+2,c+2)$ are each multiplicatively dependent, then
either
\[
	(2,b,c) \in \{(2,8,2^x 5^y - 2)
	\colon x,y \in \Z_{\geq 0},
	2^x 5^y - 2 > 8
	\}
	\cup \{(2,3,8), (2,4,8), (2,6,8)\}
\]
or
\[
	(2,b,c) \in \{
	(2, 2^z-2, 2^{2z}-2^{z+1}) 
	\colon z \in \Z_{\geq 3}
	\}
\]
or
\[
	(2,b,c) \in \{
	(2,4,14), (2,6,16)
	\}.
\]

We start by checking some small exceptional cases. Then, in Section~\ref{sec:a2_notationParity}, we introduce our notation for the triples and prove some parity conditions.
In Section~\ref{sec:a2_powers}, we prove that one of the numbers $b,c,b+2,c+2$ has to be a power of two. Finally, we finish the proof of Theorem~\ref{thm:a2} in Section~\ref{sec:a2_finish}.

\subsection{Small exceptional cases}

\begin{newlem}\label{lem:4}
The only integer triple $(2,4,c)$ with $c \geq 3$, $c\neq 4$, such that $(2,4,c)$, $(3,5,c+1)$ and $(4,6,c+2)$ are each multiplicatively dependent is $(2,4,14)$.
\end{newlem}
\begin{proof}
If $(2,4,c)$, $(3,5,c+1)$ and $(4,6,c+2)$ are each multiplicatively dependent, we have $c+1 = 3^x 5^y$ and $c+2 = 2^z 3^w$ for some nonnegative integers $x,y,z,w$. Thus, $3^x 5^y + 1 = 2^z 3^w$ and either $x=0$ or $w = 0$. 
In the first case, it follows that  $5^y + 1 = 2^z3^w$. Since $5^1 + 1 = 6$, Lemma \ref{lem:Zsigmondy} implies $y = 0$ or $y = 1$, i.e.\ $c = 5^y - 1 \in \{ 0, 4 \}$, contradicting $c \geq 3$, $c\neq 4$.

In the second case where  $w=0$, we have the equation $2^z - 1 = 3^x 5^y$. Since  $2^4 - 1 = 15$, Lemma \ref{lem:Zsigmondy} implies $z\leq 4$. But then we get $c = 2^z - 2 \in \{ -1, 0, 2, 6, 14\}$. We are not interested in $c\in \{-1,0,2\}$, the case $c=6$ does not lead to a solution, and $c=14$ gives us the triple $(2,4,14)$.
\end{proof}

\begin{newlem}\label{lem:6}
The only integer triples $(2,6,c)$ with $c \geq 3$, $c\neq 6$, such that $(2,6,c)$, $(3,7,c+1)$ and $(4,8,c+2)$ are each multiplicatively dependent\ are 
$(2,6,8)$ and $(2,6,48)$.
\end{newlem}
\begin{proof}
If $(2,6,c)$, $(3,7,c+1)$ and $(4,8,c+2)$ are each multiplicatively dependent, we have $c = 2^x 3^y$ and $c+1 = 3^z 7^w$ for some nonnegative integers $x,y,z,w$, whence  $2^x 3^y + 1 = 3^z 7^w$ and either $y=0$ or $z = 0$. 
In the first case,  $2^x + 1 = 3^z7^w$, whereby, since $2^x + 1 \equiv 2,3,5 \pmod{7}$, we have $w = 0$, and so $2^x + 1 = 3^{z}$. By Theorem \ref{thm:Catalan},  $x \in \{ 1,3 \}$ and so $c \in \{ 2,8\}$. We are not interested in $c=2$, and $c=8$ gives us our first triple.

In the second case where $z=0$, we have the equation $7^w - 1 = 2^x 3^y$. Since  $7^1 - 1 = 6$, Lemma \ref{lem:Zsigmondy} implies $w \leq 2$. Then we get $c = 7^w - 1 \in \{ 0,6,48\}$. We are not interested in $c \in\{ 0,6\}$, and $c=48$ provides our second triple.
\end{proof}

\begin{newlem}\label{lem:239}
There is no integer $c \geq 3$, $c\neq 239^2 - 1$, such that $(2,239^2 - 1,c)$, $(3,239^2,c+1)$ and $(4,239^2 + 1,c+2)$ are each multiplicatively dependent
\end{newlem}
\begin{proof}
Assume that $(2,239^2 - 1,c) = (2,2^5 \cdot 3 \cdot 5 \cdot 7 \cdot 17,c)$ and $(4,239^2 + 1,c+2) = (4,2 \cdot 13^4,c+2)$ are each multiplicatively dependent. Then we have $c = 2^x \cdot (3 \cdot 5 \cdot 7 \cdot 17)^y$ and $c+2 =  2^z \cdot 13^w$. Considering $c/2 + 1 = (c+2)/2$, we get $2^{x-1} \cdot (3 \cdot 5 \cdot 7 \cdot 17)^y + 1 = 2^{z-1} \cdot 13^w$ and $x-1 = 0$ or $z-1 = 0$. 

In the first case, we are led to the equation $(3 \cdot 5 \cdot 7 \cdot 17)^y + 1 = 2^{z-1} \cdot 13^w$. Since $(3 \cdot 5 \cdot 7 \cdot 17)^3 + 1 = 2 \cdot 13 \cdot 19 \cdot 47 \cdot 244957$, 
Lemma \ref{lem:Zsigmondy} implies $y\leq 3$, but neither of $y = 1,2,3$ leads to a solution of the equation, and we are not interested in $y = 0$ because then $c = 2$.

In the second case where $z=1$, we find that $2^{x-1} \cdot (3 \cdot 5 \cdot 7 \cdot 17)^y  = 13^w - 1$. Since $13^4 - 1 = 2^4 \cdot 3 \cdot 5 \cdot 7 \cdot 17$, Lemma \ref{lem:Zsigmondy} implies $w\leq 4$, but none of $w = 1,2, 3,4$ lead to a solution of the equation, and we are not interested in $w = 0$ because then $c+2 = 2$.

\end{proof}

\subsection{Notation and parity}\label{sec:a2_notationParity}

Assume that $(2,b,c)$, $(3,b+1,c+1)$ and $(4,b+2,c+2)$ are each multiplicatively dependent, and that $b,c$ are distinct and larger than $2$ (we do not assume $b<c$ at this point).

Then we can write
\begin{align}
(2,b,c) &= (2, 2^{y_0} q_0^{\beta_0}, 2^{z_0} q_0^{\gamma_0}),\nonumber\\
(3,b+1,c+1) &= (3, 3^{y_1} q_1^{\beta_1}, 3^{z_1} q_1^{\gamma_1}),\label{eq:a2_systemgeneral}\\ 
(4,b+2,c+2) &= (2^2, 2^{y_2} q_2^{\beta_2}, 2^{z_2} q_2^{\gamma_2}),\nonumber 
\end{align}
with integers $q_i \geq 1$ and $y_i, z_i, \beta_i, \gamma_i \geq 0$ ($i = 0,1,2$). We may assume that $q_0$ and $q_2$ are not divisible by $2$, and that $q_1$ is not divisible by $3$. 
Moreover, we may assume that $q_0,q_1,q_2>1$ (though any of the exponents might be 0).

First, we find all three consecutive multiplicatively dependent\ triples where one of $b,c$ is odd (say, $c$ is odd).

\begin{newlem}\label{lem:a2_bodd}
Let $b,c$ be distinct integers larger than $2$ such that $(2,b,c)$, $(3,b+1,c+1)$ and $(4,b+2,c+2)$ are each multiplicatively dependent. If $c$ is odd, then we have
\begin{equation}\label{eq:fam_bodd}
(2,b,c) = (2, 8,  5^x - 2)
\qquad \text{for some positive integer } x.
\end{equation} 
\end{newlem}
\begin{proof}
Assume for a moment that $b$ and $c$ are both odd. Then, with our previous  notation, we have $(b,c) = (q_0^{\beta_0}, q_0^{\gamma_0})$ and $(b+2, c+2) = (q_1^{\beta_1}, q_1^{\gamma_1})$. This is impossible by Lemma~\ref{lem:consecutivePairs2}.

Now assume that $c$ is odd and $b$ is even (in our previous notation $z_0 = z_2 = 0$ and $y_0,y_2 \geq 1$). 
Then in the triple $(3,b+1,c+1)$ the number $c+1$ is the only even number, so we must have that $(3,b+1,c+1)$ is $2$-multiplicatively dependent\ with $b+1 = 3^{y_1}$. Since we are assuming $b+1>3$, we have that $y_1 \geq 2$. 
Since $b$ and $b+2$ are even, we have either $y_0 = 1$ or $y_2 = 1$. We distinguish between these two cases and use the notation from \eqref{eq:a2_systemgeneral}.

\caselIold{Case a:} $y_0 = 1$. Recall the situation:
\begin{align*}
(2,b,c) &= (2, 2\cdot q_0^{\beta_0}, q_0^{\gamma_0}),\\
(3,b+1,c+1) &= (3, 3^{y_1}, c+1),\\ 
(4,b+2,c+2) &= (2^2, 2^{y_2} q_2^{\beta_2}, q_2^{\gamma_2}),
\end{align*}
where we clearly have $\beta_0, \gamma_0, \gamma_2 \geq 1$. Moreover, note that $\beta_2\geq 1$ because otherwise we have $(b + 1) + 1 = 3^{y_1} + 1 = 2^{y_2}$, which is impossible by Theorem \ref{thm:Catalan}. Further, we may assume that $q_0$ and $q_2$ are not divisible by $2$, and we see from $2 q_0^{\beta_0} + 1 = 2^{y_2}q_2^{\beta_2}-1 = 3^{y_1}$ that they cannot be divisible by 3. Thus, $q_0,q_2\geq 5$.

Assume for a moment that $\beta_0 = 1$. Then 
$$
2 \cdot q_0 = b = (b+2)-2 = 2^{y_2} q_2^{\beta_2} - 2 \equiv - 2 \pmod{q_2}, 
$$
which implies $q_0 \equiv - 1 \pmod{q_2}$. But then $c = q_0^{\gamma_0} \equiv \pm 1 \pmod{q_2}$, which is a contradiction because $c = (c+2)- 2 = q_2^{\gamma_2} - 2 \equiv -2 \pmod{q_2}$.

Therefore, we have $\beta_0 \geq 2$. Now consider 
\[
	(b+1) - b 
	= 3^{y_1} - 2 \cdot q_0^{\beta_0}
	= 1.
\]
From Lemma~\ref{lem:powerOf3minus2perfectpowers},  $y_1 =5$, $q_0 = 11$ and $\beta_0 =2$.
Thus, $b + 2 = 3^5 + 1 = 4 \cdot 61$, and so $q_2 = 61$. But then
\[
	(c+2) - c 
	= 61^{\gamma_2} - 11^{\gamma_0}
	= 2
\]
is a contradiction modulo 10.

\caselIold{Case b:} $y_2 = 1$. 
Recall the situation:
\begin{align*}
(2,b,c) &= (2, 2^{y_0} q_0^{\beta_0}, q_0^{\gamma_0}),\\
(3,b+1,c+1) &= (3, 3^{y_1},  c+1),\\ 
(4,b+2,c+2) &= (2^2, 2 \cdot q_2^{\beta_2}, q_2^{\gamma_2}),
\end{align*}
where $\gamma_0, \beta_2, \gamma_2 \geq 1$, and $q_0$ and $q_2$ are not divisible by $2$ or $3$, whence $q_0,q_2\geq 5$.

Assume for the moment that $\beta_0 = 0$. Then we have $b  + 1 = 2^{y_0} + 1 = 3^{y_1}$, 
which by Theorem \ref{thm:Catalan} implies $y_0 = 1$ or $y_0 = 3$, i.e.\ $b = 2$ or $b = 8$. We are not interested in $b=2$. If $b=8$, then $(4,b+2,c+2) = (4, 10, c+2)$, so $c+2 = 5^{\gamma_2}$ and we are in the family~\eqref{eq:fam_bodd}. Therefore, we may assume $\beta_0\geq 1$.

Next, assume  that $\beta_2 = 1$. Then $2 \cdot q_2 = b+2 = 2^{y_0} q_0^{\beta_0} + 2 \equiv 2 \pmod{q_0}$, which implies $q_2 \equiv 1 \pmod{q_0}$. But then $c + 2 = q_2^{\gamma_2} \equiv 1 \pmod{q_0}$, which is a contradiction because $c +2 = q_0^{\gamma_0} + 2 \equiv 2 \pmod{q_0}$.
Therefore, we have $\beta_2 \geq 2$. 

Now consider 
\[
	(b+1) - (b+2) 
	= 3^{y_1} - 2 \cdot q_2^{\beta_2}
	= -1,
\]
and as in the previous case  apply Lemma~\ref{lem:powerOf3minus2perfectpowers}; no new  
solutions accrue.
\end{proof}

In view of Lemma~\ref{lem:a2_bodd}, since triples of the shape $(2,8,5^x - 2)$ are part of family \eqref{eq:fam1}, we may from now on assume that $b$ and $c$ are even, i.e.\ in the notation of \eqref{eq:a2_systemgeneral} we have $y_0, y_2, z_0, z_2\geq 1$. 

\subsection{Proving that one of $b,c,b+2,c+2$ is a power of two}
\label{sec:a2_powers}

We first consider the cases $y_1 = 0$ or $z_1 = 0$.

\begin{newlem}\label{lem:a2_mod3}
Let $b,c$ be distinct integers larger than $2$ such that $(2,b,c)$, $(3,b+1,c+1)$ and $(4,b+2,c+2)$ are each multiplicatively dependent. If at least one of $b+1$, $c+1$ is not divisible by 3, then one of the numbers $b,c, b+2, c+2$ is a power of 2.
\end{newlem}
\begin{proof}
We use the notation from \eqref{eq:a2_systemgeneral}:
\begin{align}
(2,b,c) &= (2, 2^{y_0} q_0^{\beta_0}, 2^{z_0} q_0^{\gamma_0}),\nonumber\\
(3,b+1,c+1) &= (3, 3^{y_1} q_1^{\beta_1}, 3^{z_1} q_1^{\gamma_1}),\nonumber\\ 
(4,b+2,c+2) &= (2^2, 2^{y_2} q_2^{\beta_2}, 2^{z_2} q_2^{\gamma_2}).\nonumber 
\end{align}
In view of Lemma~\ref{lem:a2_bodd} we may assume that $y_0, z_0, y_2, z_2 \geq 1$.

We assume that $y_1 = 0$ or $z_1 = 0$, and we need to show that one of $\beta_0, \gamma_0, \beta_2, \gamma_2$ is zero.
We distinguish between two cases according to whether one or both of $y_1, z_1$ are zero.

\caselIold{Case a:} $y_1 = z_1 = 0$. 
We consider $(b+1) +1 = b+2$ and $(c+1) + 1 = c+2$:
\begin{align*}
q_1^{\beta_1} + 1 &= 2^{y_2} q_2^{\beta_2},\\
q_1^{\gamma_1} + 1 &= 2^{z_2} q_2^{\gamma_2}.
\end{align*}
Since $b+1 \neq c+1$, we have $\beta_1 \neq \gamma_1$ and so by Lemma \ref{lem:Zsigmondy} (note that $q_1 \neq 2$) $q_1^{\beta_1} + 1$ and $q_1^{\gamma_1} + 1$ have distinct sets of prime factors. Since $y_2, z_2 \geq 1$, we have either $\beta_2 = 0$ or $\gamma_2 = 0$.

\caselIold{Case b:} One of $y_1, z_1$ is $\geq 1$. Assume,
without loss of generality, $y_1 = 0$ and $z_1 \geq 1$. Since one of the numbers $b, b+1, b+2$ has to be divisible by $3$, we have either $3\mid b$ or $3\mid b+2$, i.e.\ either $3\mid q_0$ or $3\mid q_2$. Since $3$ divides $c+1$, we have $3\nmid c, c+2$. Thus, if $3 \mid q_0$, we must have $\gamma_0 = 0$, and if $3\mid q_2$, we must have $\gamma_2 = 0$. 
\end{proof}

Now we deal with the general case.

\begin{newlem}\label{lem:a2_powers}
Let $b,c$ be distinct integers larger than $2$ such that $(2,b,c)$, $(3,b+1,c+1)$ and $(4,b+2,c+2)$ are each multiplicatively dependent. Then one of the numbers $b,c, b+2, c+2$ is a power of 2.
\end{newlem}
\begin{proof}
We use the notation from \eqref{eq:a2_systemgeneral}:
\begin{align}
(2,b,c) &= (2, 2^{y_0} q_0^{\beta_0}, 2^{z_0} q_0^{\gamma_0}),\nonumber\\
(3,b+1,c+1) &= (3, 3^{y_1} q_1^{\beta_1}, 3^{z_1} q_1^{\gamma_1}),\nonumber\\ 
(4,b+2,c+2) &= (2^2, 2^{y_2} q_2^{\beta_2}, 2^{z_2} q_2^{\gamma_2}).\nonumber 
\end{align}
By Lemmata~\ref{lem:a2_bodd} and \ref{lem:a2_mod3}, we may assume that $y_0, z_0, y_1, z_1, y_2, z_2 \geq 1$.

Since exactly one of the numbers $b, b+2$ and exactly one of the numbers $c, c+2$ is $\equiv 0 \pmod 4$, we distinguish between the following three cases: ($y_0 = z_0 = 1 $ and $y_2,z_2 \geq 2$) or ($y_0,z_0 \geq 2$ and $y_2 = z_2 = 1 $) or ($y_0 = z_2 = 1$ and $y_2, z_0 \geq 2$). Note that we also have the case ($y_2 = z_0 = 1$ and $y_0, z_2 \geq 2$), but this is equivalent to the third case as we are not assuming $b<c$. We divide the proof of Lemma~\ref{lem:a2_powers}  into these three cases.

\caselI{Case 1:} $y_0 = z_0 = 1 $ and $y_2,z_2 \geq 2$. Recall that we then have 
\begin{align}
(2,b,c) &= (2, 2 \cdot q_0^{\beta_0}, 2 \cdot q_0^{\gamma_0}),\nonumber\\
(3,b+1,c+1) &= (3, 3^{y_1} q_1^{\beta_1}, 3^{z_1} q_1^{\gamma_1}),\nonumber\\ 
(4,b+2,c+2) &= (2^2, 2^{y_2} q_2^{\beta_2}, 2^{z_2} q_2^{\gamma_2}).\nonumber 
\end{align}
We consider the equations $(b) + 2 = (b+2)$ and $(c) +2 = (c+2)$ divided by 2:
\begin{align}
	q_0^{\beta_0} + 1 &= 2^{y_2 - 1}q_2^{\beta_2}, \nonumber\\
	q_0^{\gamma_0} + 1 &= 2^{z_2 - 1}q_2^{\gamma_2}. \nonumber
\end{align}
Since $b$ and $c$ are distinct, also $\beta_0$ and $\gamma_0$ are distinct. Moreover, $q_0 \neq 2$. Thus, by Lemma \ref{lem:Zsigmondy}, since $y_2 - 1, z_2 - 1 \geq 1$, we may conclude that either $\beta_2 = 0$ or $\gamma_2 = 0$.

\caselI{Case 2:} $y_0,z_0 \geq 2$ and $y_2 = z_2 = 1 $. Recall that we then have 
\begin{align}
(2,b,c) &= (2, 2^{y_0} \cdot q_0^{\beta_0}, 2^{z_0} \cdot q_0^{\gamma_0}),\nonumber\\
(3,b+1,c+1) &= (3, 3^{y_1} q_1^{\beta_1}, 3^{z_1} q_1^{\gamma_1}),\nonumber\\ 
(4,b+2,c+2) &= (2^2, 2\cdot q_2^{\beta_2}, 2\cdot q_2^{\gamma_2}).\nonumber 
\end{align}
We consider the equations $(b+2) - 2 = b$ and $(c+2) - 2 = c$ divided by 2:
\begin{align}
	q_2^{\beta_2} - 1 &= 2^{y_0 - 1} q_0^{\beta_0}, \\
	q_2^{\gamma_2} - 1 &= 2^{z_0 - 1}q_2^{\gamma_0}. \label{eq:beforeZsigmondy}
\end{align}
Since $q_2$ is not equal to two, we conclude from Lemma \ref{lem:Zsigmondy} that either $q_2 + 1$ is a power of 2 and $\{\beta_2,\gamma_2\} = \{1,2\}$, or, since $y_0-1, z_0-1 \geq 1$, that either $\beta_0 = 0$ or $\gamma_0 = 0$. Now assume that $q_2 +1$ is a power of 2, i.e. $q_2 = 2^n - 1$ for some $n\geq 2$, and, without loss of generality, $\beta_2 = 1$, $\gamma_2 = 2$. Then $c+2 = 2 \cdot (2^n - 1)^2$. But then $c+2\equiv 0,-1 \pmod{3}$, and both cases are a contradiction to $c+1 \equiv 0 \pmod{3}$. 

\caselI{Case 3:} $y_0 = z_2 = 1$ and $y_2, z_0 \geq 2$. 
Recall that we then have 
\begin{align}
(2,b,c) &= (2, 2 \cdot q_0^{\beta_0}, 2^{z_0} q_0^{\gamma_0}),\nonumber\\
(3,b+1,c+1) &= (3, 3^{y_1} q_1^{\beta_1}, 3^{z_1} q_1^{\gamma_1}),\nonumber\\ 
(4,b+2,c+2) &= (2^2, 2^{y_2} q_2^{\beta_2}, 2 \cdot q_2^{\gamma_2}).\nonumber 
\end{align}

We distinguish between three subcases: First ($y_2 \geq 3$ and  $z_0\geq 3$), then ($y_2 =2$ and $z_0 \geq 2$), and finally ($z_0 =2$ and $q_2 \geq 3$).

\caselII{Case 3.1:} $y_2 \geq 3$ and  $z_0\geq 3$.
In this case we have that $b+2 \equiv 0 \pmod{8}$ and $c \equiv 0 \pmod{8}$, and so
\begin{align}
	b + 1 = 3^{y_1} q_1^{\beta_1} &\equiv -1 \pmod{8},\label{eq:a2_case3_b1mod}\\
	c + 1 = 3^{z_1} q_1^{\gamma_1}&\equiv 1 \pmod{8}.\label{eq:a2_case3_c1mod}
\end{align}
Since $3^{y_1} \not\equiv -1 \pmod{8}$ for any $y_1$, we have that $q_1^{\beta_1} \not\equiv 1 \pmod{8}$, and so $\beta_1$ must be odd. Multiplying the congruences \eqref{eq:a2_case3_b1mod} and \eqref{eq:a2_case3_c1mod}, 
\[
	3^{y_1 + z_1} q_1^{\beta_1+\gamma_1} \equiv -1 \pmod{8},
\]
and by the same argument as before, $\beta_1 + \gamma_1$ must be odd, and so $\gamma_1$ is even. 
Then we get from \eqref{eq:a2_case3_c1mod} that $3^{z_1} \equiv 1 \pmod{8}$, and so $z_1$ is even as well. Thus, $c+1=3^{z_1}q^{\gamma_1}=:x^2$ is a square.

Now let us consider the equation $(c+1) +1 = c+2$:
\[
	x^2 + 1 = 2 \cdot q_2^{\gamma_2}.
\]
Then we obtain from Lemma~\ref{lem:Stormer} that either $\gamma_2 \leq 2$ or $(x,q_2, \gamma_2) =(239, 13,4)$.

If $(x,q_2, \gamma_2) =(239, 13,4)$,
then $c + 1 =  x^2 = 239^2$, which, by Lemma~\ref{lem:239}, does not lead to a solution.
We may therefore suppose that  $\gamma_2 \leq 2$, whereby, from $2 \cdot q_2^{\gamma_2} = c+2 \equiv 1 \pmod{3}$, necessarily $\gamma_2 = 1$.

We summarize situation:
\begin{align}
(2,b,c) &= (2, 2 \cdot q_0^{\beta_0}, 2^{z_0} q_0^{\gamma_0}),\nonumber\\
(3,b+1,c+1) &= (3, 3^{y_1} q_1^{\beta_1}, 3^{z_1} q_1^{\gamma_1} =x^2),\nonumber\\ 
(4,b+2,c+2) &= (2^2, 2^{y_2} q_2^{\beta_2}, 2 \cdot q_2),\nonumber 
\end{align}
with $y_2,z_0 \geq 3$.

We have $c+2 = 2 \cdot q_2 = 2^{z_0} q_0^{\gamma_0} +2$, which, upon division by 2, gives us 
\begin{equation}\label{eq:q2}
	q_2 = 2^{z_0 - 1} q_0^{\gamma_0} + 1. 
\end{equation}
Next, we consider the equation $(b+2) = (b) + 2$ divided by 2:
\begin{equation}\label{eq:case31a_b2}
	2^{y_2-1} q_2^{\beta_2} = q_0^{\beta_0} + 1.
\end{equation}
Substituting \eqref{eq:q2} in \eqref{eq:case31a_b2}, we obtain
\begin{equation}\label{eq:q0}
	2^{y_2-1} (2^{z_0 - 1} q_0^{\gamma_0} + 1)^{\beta_2} = q_0^{\beta_0} +1.
\end{equation}
Computing the 2-adic valuation of both sides of \eqref{eq:q0}, we find that
\begin{equation*}\label{eq:val}
y_2 - 1 = v_2(q_0^{\beta_0} +1).
\end{equation*}
Since $y_2-1\geq 2$, $\beta_0$ must be odd, and, since Lemma~\ref{lem:v2_oddexpo}  implies that $v_2(q_0^{\beta_0} +1) = v_2(q_0 + 1)$, 
\begin{equation}\label{eq:val2}
y_2 - 1 
= v_2(q_0 +1) 
\leq \log_2(q_0 +1).
\end{equation}
Considering equation~\eqref{eq:q0} modulo $q_0$,  we see that 
\[
	2^{y_2 - 1} \equiv 1 \pmod{q_0}.
\]
This implies that $2^{y_2 - 1} \geq q_0 + 1$, and so $y_2 - 1 \geq \log_2(q_0 + 1)$. Therefore, we have equality in \eqref{eq:val2}, which implies $q_0 + 1 = 2^{y_2-1}$, i.e.\ 
\[
	q_0 = 2^{y_2-1}-1.
\]
Since $3\nmid q_0 = 2^{y_2-1}-1$, we have that $y_2-1$ is odd. Thus, $y_2$ is even, and 
$q_0 = 2^{y_2 - 1} - 1 \equiv 1 \pmod{3}$.
Next, note that $2 q_2 = c+2 \equiv 1 \pmod{3}$, so $q_2 \equiv{-1} \pmod 3$. 
Moreover,
$1 \equiv b+ 2 = 2^{y_2}q_2^{\beta_2} \equiv 1 \cdot (-1)^{\beta_2} \pmod{3}$, so $\beta_2$ is even. 
Thus, $2^{y_2 - 2} q_2^{\beta_2} =: x_1^2$ is a square. 

Substituting this in \eqref{eq:case31a_b2}, we obtain the equation
\[
	2 \cdot x_1^2 - 1 = q_0^{\beta_0}.
\]
Lemma \ref{lem:two_squares_minus_one} implies $\beta_0 \leq 2$ (the exceptional solution leads to $x_1 = 78$, $q_0 = 23$, $\beta_0 = 3$, but then $23$ is not of the shape  $q_0 = 2^{y_2-1}-1$).

Since $\beta_0$ is odd (see argument above \eqref{eq:val2}), we get $\beta_0 = 1$. 
Thus, $b = 2 \cdot(2^{y_2-1}-1)$ and $b+2 = 2^{y_2}$.

\caselII{Case 3.2:} $y_2 = 2$ and $z_2 \geq 2$.
Recall the situation:
\begin{align}
(2,b,c) &= (2, 2 \cdot q_0^{\beta_0}, 2^{z_0} q_0^{\gamma_0}),\nonumber\\
(3,b+1,c+1) &= (3, 3^{y_1} q_1^{\beta_1}, 3^{z_1} q_1^{\gamma_1}),\nonumber\\ 
(4,b+2,c+2) &= (2^2, 4 \cdot q_2^{\beta_2}, 2 \cdot q_2^{\gamma_2}).\nonumber 
\end{align}

We have $2 \cdot q_2^{\gamma_2} = c+2\equiv 1 \pmod{3}$, which implies $q_2^{\gamma_2} \equiv - 1 \pmod{3}$. Thus, $q_2 \equiv - 1 \pmod{3}$ and $\gamma_2$ is odd. Moreover, we have $b+2 = 4 \cdot q_2^{\beta_2} \equiv q_2^{\beta_2} \equiv 1 \pmod{3}$, and so $\beta_2$ has to be even.

Now we consider the equations $b = (b+2) - 2$ and $c = (c+2)- 2$ divided by $2$:

\begin{align}
q_0^{\beta_0} = 2 q_2^{\beta_2} - 1, \label{eq:b}\\
2^{z_0-1} q_0^{\gamma_0} =  q_2^{\gamma_2} - 1. \label{eq:c}
\end{align}
The right hand side of \eqref{eq:c} is divisible by $q_2 - 1$. If $q_2 - 1$ is a power of 2, then we have $q_2 = 2^n + 1$. Otherwise, let $p>2$ be a prime divisor of $q_2 - 1$ and write $q_2 = kp + 1$. Then, by \eqref{eq:c}, $p$ is also a prime divisor of $q_0$. Substituting into \eqref{eq:b}, we obtain
\[
	q_0^{\beta_0} = 2 (kp + 1)^{\beta_2} - 1.
\]
Since $p\mid q_0$, we obtain a contradiction modulo $p$. Thus, $q_2 = 2^n + 1$. 
Since $\gamma_2$ is odd, we see from \eqref{eq:c} that $z_0 -1 = v_2(q_2^{\gamma_2}-1) = v_2((2^n + 1)^{\gamma_2} - 1) = n$, and so $n=z_0-1$ and 
\[
	q_2 =2^{z_0-1} +1.
\]

Next, we take \eqref{eq:c} to the power of $\beta_0$, and substitute $q_0^{\beta_0} = 2 q_2^{\beta_2} - 1$:
\[
	(2^{z_0-1})^{\beta_0} (2 q_2^{\beta_2} - 1)^{\gamma_0} = (q_2^{\gamma_2} - 1)^{\beta_0}.
\]
Considering the equation modulo $q_2 = 2^{z_0-1} + 1$, 
\[
	(-1)^{\beta_0} (-1)^{\gamma_0} \equiv (-1)^{\beta_0} \pmod{q_0},
\]
which implies $(-1)^{\gamma_0} \equiv 1 \pmod{q_0}$, and so $\gamma_0$ is even.

Also, since we are assuming $y_1, z_1 \neq 0$, we have that $3\nmid q_2 =2^{z_0 - 1} + 1$, and so $z_0-1$ is even.

Therefore, the left hand side of \eqref{eq:c} is a square, and hence, from Theorem \ref{thm:Catalan}, $\gamma_2 = 1$. 
Then \eqref{eq:c} becomes
\[
	2^{z_0-1} q_0^{\gamma_0}
	= (2^{z_0-1} +1)^1 -1,
\]
and so $\gamma_0 = 0$.

\caselII{Case 3.3:} $z_0 = 2$ and $y_2\geq 3$.
Recall the situation:
\begin{align}
(2,b,c) &= (2, 2 \cdot q_0^{\beta_0}, 4 \cdot q_0^{\gamma_0}),\nonumber\\
(3,b+1,c+1) &= (3, 3^{y_1} q_1^{\beta_1}, 3^{z_1} q_1^{\gamma_1}),\nonumber\\ 
(4,b+2,c+2) &= (2^2, 2^{y_2} q_2^{\beta_2}, 2 \cdot q_2^{\gamma_2}).\nonumber 
\end{align}
We have $4 \cdot q_0^{\gamma_0} = c \equiv -1 \pmod{3}$, which implies $q_0^{\gamma_0} \equiv -1 \pmod{3}$, and so $q_0 \equiv - 1 \pmod{3}$ and $\gamma_0$ is odd. On the other hand, $2 \cdot q_0^{\beta_0} = b \equiv -1 \pmod{3}$, so $\beta_0$ is even. Thus, $b = 2 \cdot q_0^{\beta_0} \equiv 2 \cdot 1 = 2 \pmod{8}$, which is a contradiction because $b = (b+2) - 2 \equiv -2 \pmod{8}$.
\end{proof}

\subsection{Finishing the proof of Theorem~\ref{thm:a2}}
\label{sec:a2_finish}

Now we finish the proof of Theorem~\ref{thm:a2} using the notation from \eqref{eq:a2_systemgeneral}.
By Lemma \ref{lem:a2_powers}, we may assume, without loss of generality, that either $b$ or $b+2$ is a power of two. We distinguish between these two cases.

\caselI{Case a:} $b$ is a power of two. Then we are in the following situation:
\begin{align}
(2,b,c) &= (2, 2^{y_0} , c),\nonumber\\
(3,b+1,c+1) &= (3, 3^{y_1} q_1^{\beta_1}, 3^{z_1} q_1^{\gamma_1}),\nonumber\\ 
(4,b+2,c+2) &= (2^2, 2 \cdot (2^{y_0-1} + 1), 2^{z_2} \cdot (2^{y_0-1} + 1)^{\gamma_2}).\nonumber 
\end{align}
Note that $b>2$, so $y_0 \geq 2$. If $y_0 = 2$, we are done with Lemma~\ref{lem:4}. If $y_0 = 3$, we get the family \eqref{eq:fam1}. Therefore, we may assume $y_0 \geq 4$.
Moreover, in view of Lemma~\ref{lem:a2_bodd}, we may assume that $z_2 \geq 1$.

Considering the equation $-1 = b - (b+1) = 2^{y_0} - 3^{y_1} q_1^{\beta_1}$,  Lemma~\ref{lem:23q} implies  that $\beta_1 \leq 2$.
Note that $\beta_1 = 0$ is impossible via Theorem \ref{thm:Catalan}. Now assume for a moment that $\beta_1 = 2$. 
Then we have $ 3^{y_1} q_1^{2} =  b+1 \equiv 1 \pmod 8$, which implies 
$3^{y_1} \equiv 1 \pmod{8}$, so $y_1$ is even as well. But then $b+1 = 2^{y_0}+1$ is a square, contradicting Theorem \ref{thm:Catalan}. Therefore, we have $\beta_1 = 1$. Then, $b+1 = 2^{y_0}+1 = 3^{y_1} q_1$ implies 
$q_1 = 3^{-y_1 } (2^{y_0}+1)$ and the equation $(c+2)- 1 = c+1$ becomes
\[
	2^{z_2} \cdot (2^{y_0-1} + 1)^{\gamma_2} - 1
	= 3^{z_1} q_1^{\gamma_1}
	= 3^{z_1} (3^{-y_1} (2^{y_0}+1))^{\gamma_1}
	= 3^{z_1 - y_1 \gamma_1} (2^{y_0}+1)^{\gamma_1}.
\]
Note that we have $y_0 - 1 \geq 3$, $z_2 \geq 1$, $\gamma_2 \geq 0$ and $\gamma_1 \geq 0$.
Then Proposition~\ref{prop:sit1} implies
\[
	(y_0-1, z_2, \gamma_2, \gamma_1) \in 
	\{
		(y_0-1, 1, 0, 0),
		(y_0-1, 1, 1, 1),
		(y_0-1, 2, 0, 0)
	\}.
\]
The first type of solution yields $c+2 = 2^1 \cdot (2^{y_0-1} + 1)^0 = 2$, the second type of solution yields $c+2 = 2^1 \cdot (2^{y_0-1} + 1)^1 = b+2$, and the third type of solution yields $c+2 = 2^2 \cdot (2^{y_0-1} + 1)^0 = 4$.

\caselI{Case b:} $b+2$ is a power of two. Then we are in the following situation:
\begin{align}
(2,b,c) &= (2, 2 \cdot (2^{y_2 - 1} - 1), 2^{z_0} \cdot (2^{y_2 - 1} - 1)^{\gamma_0}),\nonumber\\
 (3,b+1,c+1) &= (3, 3^{y_1} q_1^{\beta_1}, 3^{z_1} q_1^{\gamma_1}),\nonumber\\ 
 (4,b+2,c+2) &= (2^2, 2^{y_2}, c+2).\nonumber 
\end{align}
Note that $b+2>4$, so $y_2 \geq 3$. If $y_2 = 3$, we have $b=6$, and we are done with Lemma~\ref{lem:6}. Therefore, we my assume $y_0 \geq 4$. 
Moreover, in view of Lemma~\ref{lem:a2_bodd} we may assume that $z_0 \geq 1$.

Considering the equation $1 = (b+2) - (b+1) = 2^{y_2} - 3^{y_1} q_1^{\beta_1}$, we get from Lemma~\ref{lem:23q} that $\beta_1 \leq 2$.
Note that $\beta_1 = 0$ is impossible via Theorem \ref{thm:Catalan}. 
Now assume for a moment that $\beta_1 = 2$. Then 
$ 3^{y_1} q_1^{2} =  b+1 \equiv - 1 \pmod 8$, which is impossible. Therefore, we have $\beta_1 = 1$. 
Then, $b+1 = 2^{y_2}-1 = 3^{y_1} q_1$ implies 
$q_1 = 3^{-y_1 } (2^{y_2} - 1)$ and the equation $(c)+1 = (c+1)$ becomes
\[
	2^{z_0} \cdot (2^{y_2 - 1} - 1)^{\gamma_0} + 1
	= 3^{z_1} (3^{-y_1} (2^{y_2}-1))^{\gamma_1}
	= 3^{z_1 - y_1 \gamma_1} (2^{y_2}-1)^{\gamma_1}.
\]
Note that we have $y_2 - 1\geq 3$, $z_0 \geq 1$, $\gamma_0 \geq 0$ and $\gamma_1 \geq 0$.
Then Proposition~\ref{prop:sit2} implies 
\begin{align*}
	(y_2-1, z_0, \gamma_0, \gamma_1) \in 
	\{
		&(y_2-1, y_2 + 1, 1, 2),
		(y_2-1, 1, 1, 1),
		(y_2-1, 1, 0, 0),\\
		&(y_2-1, 3, 0, 0),
		(3, 2, 0, 1, -1)
	\}.
\end{align*}
The first type of solution leads to $c =  2^{y_2 + 1} \cdot (2^{y_2-1} - 1)^1$, which corresponds to family \eqref{eq:newFamily}. 
The second type of solution yields $c = 2^1 \cdot (2^{y_2-1} - 1)^1 = b$ and the third one yields $c = 2^1 \cdot (2^{y_2-1} - 1)^0 = 2$, neither of which  are of interest. The fourth type of solution leads to
$c = 2^3 \cdot (2^{y_2-1} - 1)^0 = 8$, which corresponds to family \eqref{eq:fam1}.
Finally, the fifth solution leads to $b = 2^{y_2} - 2 = 14$ and $c = 2^2 \cdot (2^3 - 1)^0 = 4$, corresponding to the triple $(2,4,14)$.
This completes the proof of Theorem~\ref{thm:a2}.

\section{On the equation $x^2-2=y^n$}\label{sec:x2-2}

In order to prove Theorem~\ref{thm:3x2md}, we will need rather strong bounds on the solutions of the equation $x^2-2 = y^n$. 
As mentioned in the introduction, this equation has been studied intensively, but is still not solved completely. In the following proposition, we will collect  results of Chen \cite{Chen2012}, together with an upper bound on the exponent $n$, which we will derive via  an application of Laurent's lower bounds for linear forms in two logarithms \cite{Laurent2008}. A similar result was achieved by Bugeaud, Mignotte, and Siksek (see \cite{BugeaudMignotteSiksek2006}), published  without details.

\begin{newprop}\label{prop:main}
Assume that $n \geq 2, x\geq 2$ and $y$ are integers with
\begin{equation}\label{eq:mainsquareeq}
	x^2 - 2 = y^n.
\end{equation}
Then the following hold:
\begin{enumerate}[(i)]
 \item $n$ is prime. 
  \item $41 \le n \leq 1237$. \label{it:nsmall}
  \item $n \equiv 13,17,19,23 \pmod{24}$. 
  \item $y \equiv -1 \pmod{Q}$, where $Q:=\prod_{q \in S}{q}$.  \label{it:ymodQ}
  \item $y>10^{102}$. 
\end{enumerate}
The set $S$ in (\ref{it:ymodQ}) is given by
$$
\begin{array}{l}
S= \left\{5, 7, 11, 13, 19, 23, 29, 31, 37, 41, 61, 67, 73, 89, 113, 127, 137, 149, 181, 191, 193, 197, 223, \right. \\
  233, 251, 257, 349, 373, 379, 421, 457, 461, 521, 547, 599, 617, 661, 677, 701, 761, 769, 811,  829, \\
\left.  881, 883, 953 \right\}.
\end{array}
$$

\end{newprop}

We begin by stating a number of known results on equation (\ref{eq:mainsquareeq}). For small values of the exponent, we have the following.

\begin{knownprop}[{\cite[Lemma 15.7.3]{Cohen2007II}}]\label{prop:SiksekSmallN}
Let $n$ be prime with $2\leq n \leq 37$.
Then equation (\ref{eq:mainsquareeq})
has no integer solutions $(x,y)$ with $x\geq 2$.
\end{knownprop}

For general prime $n$, we have the next result by Chen.

\begin{knownprop}[{\cite[Theorem 5]{Chen2012}}]\label{prop:Chen_congruences}
Let $n$ be prime and assume that  equation (\ref{eq:mainsquareeq})
has an integer solution $(x,y)$ with $x\geq 2$. Then $n \equiv 13, 17, 19, 23 \pmod{24}$.
\end{knownprop}

Moreover, Chen proved the following (see \cite[Corollary 25 and its proof]{Chen2012}).

\begin{knownprop}[{\cite[Corollary 25]{Chen2012}}]\label{prop:Chen}
Let $S=\{5, \ldots , 953\}$ be the set of primes defined in Proposition~\ref{prop:main}, and $Q = \prod_{q \in S} q$.
Let $p$ be prime. 
If equation (\ref{eq:mainsquareeq}) has a solution in integers $x,y$, 
then
\[
	y \equiv -1 \pmod{Q}.
\]
In particular, if $x\geq 2$, then $y > 10^{102}$.
\end{knownprop}

\subsection{Factorization of $x^2 - 2$}\label{sec:sub:fact}

Assume that the equation
\begin{equation}\label{eq:factinZ2}
	y^n = x^2 - 2 = (x+\sqrt{2})(x-\sqrt{2})
\end{equation}
is satisfied for some integers $n\geq 2$, $x\geq 2$, and $y\geq 2$. 
First, observe that $x$ and $y$ are odd since otherwise we obtain a contradiction modulo $4$.
Note that $\Z[\sqrt{2}]$ is a unique factorization domain, with unit group  given by $\Z[\sqrt{2}]^\times = \{\pm(\sqrt{2} + 1)^r, r \in \Z\}$.

Assume for the moment that the factors on the right hand side of \eqref{eq:factinZ2} 
have a non-unit common factor  in $\Z[\sqrt{2}]$. Then this factor also divides $(x+\sqrt{2})-(x-\sqrt{2}) = 2\sqrt{2}$, whence $y$ is even, a contradiction. Therefore, $(x+\sqrt{2})$ and $(x-\sqrt{2})$ are coprime in $\Z[\sqrt{2}]$, whereby, 
since the right hand side of \eqref{eq:factinZ2} is an $n$-th power, the prime factors in $(x+\sqrt{2})$ and $(x-\sqrt{2})$ must occur as $n$-th powers as well. In particular, we have $x+\sqrt{2} = (\sqrt{2} + 1)^r (a + b \sqrt{2})^n$ for some integers $r, a, b$. Conjugating, we obtain
\begin{equation}\label{eq:factorization}
\begin{cases}
x+\sqrt{2}=(\sqrt{2} + 1)^r(a+b\sqrt{2})^n,  \\
x-\sqrt{2}=(\sqrt{2} - 1)^r(a-b\sqrt{2})^n,
\end{cases}
\end{equation}
and
\[
	y = a^2 - 2b^2.
\]
We may thus assume that $a$ and $b$ are integers with $a\geq 1$, $b\neq 0$, $a+b\sqrt{2}>0$ and $a- b\sqrt{2} >0$, and  that $-(n-1)/2 <r \le (n-1)/2$, by including all $n$-th powers of the fundamental unit in the factor $(a + b\sqrt{2})^n$.
We have the following.
\begin{knownlem}[{\cite[Proposition 15.7.1]{Cohen2007II}}]\label{lem:r}
If $n$ is a prime $\geq 5$, then with the above notation we have $r = \pm 1$.
\end{knownlem}

\subsection{Proof of Proposition~\ref{prop:main}: an upper bound for $n$}

Suppose now that we have a solution in integers $x, y, n$ to
\begin{equation}\label{eq:mainsquare2}
	x^2 - 2 = y^n,
\end{equation}
with $n\geq 2$ and  $x \geq 2$. We may assume, without loss of generality, from Proposition \ref{prop:SiksekSmallN}, that $n \geq 41$, and from Proposition \ref{prop:Chen}, that $y > 10^{102}$.

From (\ref{eq:factorization}) and Lemma \ref{lem:r}, we have
\begin{equation*}\label{eq:linform-mult}
\left| (\sqrt{2} + 1)^{\kappa} \left(\frac{a-b\sqrt{2}}{a+b\sqrt{2}}\right)^n -1  \right|
= \frac{2\sqrt{2}}{x+\sqrt{2}}
< 0.5,
\end{equation*}
where $\kappa = \pm 2$. Since for any $z \in \R$ with $|z-1| \leq \frac{1}{2}$ one has $|\log{z}|< 2 |z-1|$, we obtain
\begin{equation}\label{eq:linform}
	|\Lambda|:=	
	\left| 
		n \log \left(\frac{a-b\sqrt{2}}{a+b\sqrt{2}}\right) 
		\pm 2  \log (\sqrt{2} + 1) 
	\right|
	< \frac{4\sqrt{2}}{x+\sqrt{2}}.
\end{equation}

Next, we check that $(a-b\sqrt{2})/(a+b\sqrt{2})$ and $(\sqrt{2} + 1)$ are multiplicatively independent. Since $(\sqrt{2} + 1)$ is a unit, we only need to verify that $(a-b\sqrt{2})/(a+b\sqrt{2})$ is not. But this is clear from the factorization argument in Section~\ref{sec:sub:fact}.

We begin by applying Lemma \ref{cor:Laurent} to obtain a lower bound for $|\Lambda| = |b_2 \log \alpha_2 - b_1 \log \alpha_1|$ with
\begin{equation}\label{eq:Lambda-variables}
	b_2 = n, \quad
	\alpha_2 = \frac{a-b\sqrt{2}}{a+b\sqrt{2}}, \quad
	b_1 = 2, \quad
	\alpha_1 = \sqrt{2} \pm  1.
\end{equation}
Supposing that $a>0$ and $b<0$, it follows that $\alpha_2 > 1$ and hence from (\ref{eq:linform}), $\alpha_1 =\sqrt{2}+1$. We have $D=2$. 
Since $\h(\alpha_1) = \h(\sqrt{2} + 1) = 1/2 \cdot \log(\sqrt{2} + 1)$, we can take $\log A_1=1/2$. Further, 
\begin{align*}
	\alpha_2 
	= \frac{a-b\sqrt{2}}{a+b\sqrt{2}}
	= \left( \frac{x- \sqrt{2}}{x + \sqrt{2}} \cdot  
		\left( \frac{\sqrt{2} + 1}{\sqrt{2} - 1} \right)
	\right)^{1/n} 
	= \left( \frac{x- \sqrt{2}}{x + \sqrt{2}} \cdot  
		(\sqrt{2} + 1)^{2}
	\right)^{1/n}.
\end{align*}
Since $(x+\sqrt{2})/(x-\sqrt{2}) < \sqrt{2} + 1$, we obtain
\begin{equation*}
	\alpha_2
	< \left( (\sqrt{2} + 1)^{3 }\right)^{1/n}
	= (\sqrt{2} + 1)^{3/n },
\end{equation*}
and therefore
\begin{equation}\label{eq:logalpha2}
	\log \alpha_2 < \frac{3}{n} \log(\sqrt{2} + 1).
\end{equation}
In order to estimate the height of $\alpha_2 = (a-b\sqrt{2})/(a+b\sqrt{2})$, we note  that the conjugate is $(a+b\sqrt{2})/(a-b\sqrt{2}) = 1/\alpha_2$, and therefore the minimal polynomial of $\alpha_2$ is
\[
	X^2 - 2 \cdot \frac{a^2 + 2b^2}{	a^2 - 2b^2} \cdot X + 1
	= X^2 - 2 \cdot \frac{a^2 + 2b^2}{y} \cdot X + 1.
\]
Therefore, $\alpha_2$ is a root of the polynomial $f(X) = y X^2 - 2(a^2 + 2 b^2) X + y$, whence
\[
	\h(\alpha_2)
	\leq \frac{1}{2} (\log y + \log \alpha_2).
\]
We may therefore take
$$
\log A_2 =  \frac{1}{2} \left( \log y + \frac{3}{n} \log(\sqrt{2} + 1) \right).
$$
It follows that
$$
b^\prime = \frac{2}{\log y + \frac{3}{n} \log(\sqrt{2} + 1)} + n,
$$
whereby, crudely from $n \geq 41$ and $y > 10^{102}$, we have
$$
n< b^\prime < n+1.
$$
If we suppose that $n > 8000$, it follows that 
$$
 \log b^\prime + 0.38 > 9,
 $$
whence Lemma \ref{cor:Laurent} implies that
$$
\log |\Lambda| \geq -81.2 \left(  \log (n+1) + 0.38 \right)^2 \left( \log y + \frac{3}{n} \log(\sqrt{2} + 1) \right).
$$
Again appealing to $n > 8000$ and $y > 10^{102}$, we thus have
$$
\log |\Lambda| \geq -89 \log^2 n \log y.
$$

In the other direction, we will  combine this with an upper bound from \eqref{eq:linform}. We have
\begin{align*}
	|\Lambda|
	< \frac{4\sqrt{2}}{x+\sqrt{2}}
	< \frac{4\sqrt{2}}{y^{n/2}},
\end{align*}
which implies
\begin{equation} \label{upper}
	\log |\Lambda| 
	< \log (4\sqrt{2}) - \frac{n}{2} \log y.
\end{equation}
We thus have
$$
 n  < 178 \log^2 n + \frac{2 \log (4\sqrt{2}) }{\log y}.
$$
It follows, from $y > 10^{102}$, that $n < 17000$.

To sharpen this result, we will argue more carefully, appealing to Theorem~\ref{thm:Laurent}.
For any $\rho >1$ we have
\[
	(\rho - 1) \log \alpha_1 + 2 D \h (\alpha_1)
	= (\rho + 1) \log(\sqrt{2} + 1).
\]
Thus, if $\rho \geq 2$, we can set 
\begin{equation}\label{eq:a1}
	a_1 = (\rho + 1) \log(\sqrt{2} + 1)
	\geq \max\{ 1, (\rho-1) \log \alpha_1 + 2 D \h(\alpha_1)\}.
\end{equation}
To choose the parameter $a_2$, note that,  for any $\rho > 1$ we have, using the above estimate and \eqref{eq:logalpha2},
\begin{align*}
	(\rho - 1)\log \alpha_2 + 2D \h(\alpha_2)
	&< (\rho - 1)\log \alpha_2 + 2 (\log y + \log \alpha_2)\\
	&= (\rho + 1) \log \alpha_2 + 2 \log y\\
	&< (\rho + 1)  \frac{3}{n} \log(\sqrt{2} + 1) + 2 \log y\\
	&< \left(
		2+\frac{(\rho+1)\log(\sqrt{2}+1)}{\log{y_0}}\cdot \frac{3}{n}  
		\right) \cdot \log{y},
\end{align*}
where $y_0> 1 $ is any lower bound for $y$. 
 Let $n_0\geq 41$ be any fixed lower bound for $n$.
Thus we can choose
\begin{equation}\label{eq:a2-nprime}
	a_2 = 
	\left(2+\frac{(\rho+1)\log(\sqrt{2}+1)}{\log{y_0}}\cdot \frac{3}{n_0}  \right) \cdot \log{y}.
\end{equation}
Now we set
\begin{equation}\label{eq:param1}
	\mu = 0.55, \quad
	\rho = 26, \quad
	y_0 = 10^{102} \; \; \mbox{ and } \; \; 
	n_0 = 1289.
\end{equation}
Then we can compute
\begin{align*}
\sigma = \frac{1 + 2\mu - \mu^2}{2} = 0.89875, \quad
\lambda = \sigma \log \rho =0.89875 \log 26,
\end{align*}
and obtain
$$
a_1 = (\rho + 1) \log(\sqrt{2} + 1) = 27 \log(\sqrt{2} + 1)
$$
and
\begin{equation} \label{turtle}
2.00023 \log y < a_2 < 2.00024 \log y.
\end{equation} 
Since $y\geq y_0$, we find from (\ref{eq:a2-nprime}) that $a_2 > 469.78$.
Recall that $b_1 = 2 < n$ and $b_2 = n$. We thus  obtain
\begin{equation} \label{Wuggles}
\log\left(\frac{b_1}{a_2}+\frac{b_2}{a_1}\right)
	<  \log\left(\frac{2}{469.78} + \frac{n}{ 27 \log(\sqrt{2} + 1)} \right)
	=\log(n)- \bsubtr,
\end{equation}	
where, from $n \geq 1289$, 
\[ 
	3.1694 < \bsubtr < 3.1696.
\]
Now we can set 
\begin{align*}
	&h 
	= \max\{ 2 (\log n- \bsubtr + \log \lambda + 1.75) + 0.06, \lambda, \log 2 \}
	= 2 \log n - \hsubtr,
\end{align*}
where
\[
	\hsubtr = 2\bsubtr - 2\log \lambda - 3.56
\]
and we have used the fact that 
$$
2 \log n - \hsubtr \geq 2 \log n_0 - \hsubtr > \lambda >\log 2.
$$
We thus have
$$
0.6300 < \hsubtr < 0.6305.
$$

Next, we can easily check that $a_1a_2 \geq \lambda^2$ holds.
Moreover, we have 
\[
	h \geq 2 \log n_0 - \hsubtr := h_0 > 13.6927,
\]
and therefore
\[
	H \geq \frac{h_0}{\lambda} + \frac{1}{\sigma} := H_0 > 5.7887,
\]
\[
	\omega \leq \omega_{\max}
	:= 2 \left( 1 + \sqrt{1 + \frac{1}{4 H_0^2}} \right)
	< 4.0075,
\]
and
\[
	\theta \leq \theta_{\max}
	:= \sqrt{1 + \frac{1}{4 H_0^2}} + \frac{1}{2H_0}
	< 1.0901.
\]
We therefore have
$$
C \leq C_{\max}
	= \frac{\mu}{\lambda^3 \sigma} \left(
		\frac{\omega_{\max}}{6} + \frac{1}{2} \sqrt{
			\frac{\omega_{\max}^2}{9} 
			+ \frac{8 \lambda \omega_{\max}^{5/4} \theta_{\max}^{1/4}}{3 \sqrt{a_1 a_{2}} H_0^{1/2}}
			+ \frac{4}{3} \left( \frac{1}{a_1} + \frac{1}{a_{2}} \right)
			\frac{\lambda \omega_{\max}}{H_0}		
		}
	\right)^2,
$$
where, from (\ref{turtle}), 
$$
a_1 =22 \log(\sqrt{2} + 1) \; \; \mbox{ and } \; \;  a_2 > 2.00023 \log y > 2.00023 \cdot 102 \log 10.
$$
We thus have that
$$
C \leq C_{\max} < 0.0471
$$
and
\begin{align*}
C' \leq C_{\max}'
	:= \sqrt{\frac{C_{\max} \sigma \omega_{\max} \theta_{\max}}{\lambda^3 \mu}}
	<  0.1158.
\end{align*}
Now we estimate the summands (after dividing out by $\log y$) in the lower bound for $|\Lambda|$:
\begin{equation} \label{eq:LaurentBound1}
\frac{C \left(h + \frac{\lambda}{\sigma}\right)^2 a_1 a_2}{\log{y}}
	<  C_{\max} \left(h+\frac{\lambda}{\sigma}\right)^2 a_1 \cdot 2.00024
	< 2.2420 \left(h+\frac{\lambda}{\sigma}\right)^2, 
\end{equation} 
\begin{equation} \label{eq:LaurentBound2}
\frac{\sqrt{\omega\theta}\left(h+\frac{\lambda}{\sigma}\right)}{\log{y}}
	\leq \frac{\sqrt{\omega_{\max}\theta_{\max}}\left(h+\frac{\lambda}{\sigma}\right)}{\log{y_0}}	
	<0.0089 \left(h+\frac{\lambda}{\sigma}\right),
\end{equation}
\begin{align}
\frac{\log\left( C'\left(h+\frac{\lambda}{\sigma}\right)^2a_1a_2 \right)}{\log{y}}
	&\leq  \frac{\log\left( C'_{\max} \left(h+\frac{\lambda}{\sigma}\right)^2 a_1 a_2 \right)}{\log{y}} \nonumber\\
	& \leq  \frac{\log C'_{\max} + 2 \log \left(h+\frac{\lambda}{\sigma}\right) + \log a_1 + \log (2.00024)}{\log y_0} + \frac{\log \log y}{\log y} \nonumber \\
	& \leq \frac{\log C'_{\max} + \log a_1 + \log (2.00024)}{\log y_0} + \frac{ 2 }{\log y_0}\log \left(h+\frac{\lambda}{\sigma}\right) + \frac{\log \log y_0}{\log y_0} \nonumber \\
	&<  0.0306 + 0.0086 \log \left(h+\frac{\lambda}{\sigma}\right).
	\label{eq:LaurentBound3}
\end{align}
The bounds (\ref{eq:LaurentBound1})--(\ref{eq:LaurentBound3}), together  with  Theorem~\ref{thm:Laurent},  imply the inequality
\[
	\frac{\log |\Lambda|}{\log y}
	> -2.2420 \left(h+\frac{\lambda}{\sigma}\right)^2 - 0.0089 \left(h+\frac{\lambda}{\sigma}\right)  - 0.0306 - 0.0086 \log \left(h+\frac{\lambda}{\sigma}\right).
\]
Denoting 
$$
T(h) = 2.2420 \left(h+\frac{\lambda}{\sigma}\right)^2 + 0.0089 \left(h+\frac{\lambda}{\sigma}\right)  + 0.0306 + 0.0086 \log \left(h+\frac{\lambda}{\sigma}\right),
$$
and combining this with inequality (\ref{upper}), we thus have
\begin{equation} \label{final}
	n 
	< \frac{2 \log (4\sqrt{2})}{\log y} - 2 \frac{\log |\Lambda| }{\log y}
	< \frac{\log 32}{\log y_0} + 2 T(h).
\end{equation}
From the fact that
$$
h < 2 \log n - 0.6300,
$$
a short calculation verifies that  inequality (\ref{final}) fails for each integer $n$ with $1289 \leq  n < 17000$. It therefore follows that $n < 1289$, whence, from Propositions \ref{prop:SiksekSmallN} and
\ref{prop:Chen_congruences}, we have that $n$ is necessarily prime and that $n \leq 1237$ as claimed.
This completes the proof of  
Proposition~\ref{prop:main}~\eqref{it:nsmall}.

\begin{rem*}
If  we appeal somewhat carefully to a more general theorem of Laurent \cite[Theorem 1]{Laurent2008} and enlarge the set $S$ described in Proposition~\ref{prop:main}, it is possible to slightly sharpen our upper bound upon the exponent $n$ in 
Proposition~\ref{prop:main}, to roughly $n < 1000$.
\end{rem*}

\section{Solving the equations from Problem \ref{probl:solveEquations}}\label{sec:problem_equations}

\begin{newprop}\label{prop:eq48}
The equation
\begin{equation}\label{eq:48}
	((d^x + 1)^s + 1)^t - d^y  = 2
\end{equation}
has no solution in integers $d,x,y,s,t$ with $x \geq 1$ and ${d,s,t,y} \geq 2$.
\end{newprop}
\begin{proof} Let $d,x,y,s,t$ be integers with $x \geq 1$,   ${d,s,t,y} \geq 2$, satisfying \eqref{eq:48}.
We may assume that $d$ is odd, since otherwise ${t, y} \geq 2$ leads to a contradiction  modulo $4$.

If $y$ is even, then Theorem~\ref{thm:Nagell} implies $d = 5$ and $(d^x + 1)^s + 1 = 3$, a contradiction. We may thus suppose that $y$ is odd and, 
using the binomial theorem,  rewrite \eqref{eq:48} as
\begin{equation}\label{eq:48rew}
	 \sum_{i=1}^t {t \choose i} (d^x + 1)^{si}
	 = d^y + 1.
\end{equation}
Since the left hand side of \eqref{eq:48rew} is divisible by $(d^x + 1)$, we have $d^x+1 \mid d^y+1$. By Lemma~\ref{lem:divis_axplus1}, this implies that $x\mid y$ and since we are assuming that $y$ is odd, $x$ is odd as well. 

Now we consider the 2-adic valuations of the expressions in \eqref{eq:48rew}. 
For the left hand side we have
\begin{equation*}
	v_2\left( 
		\sum_{i=1}^t {t \choose i} (d^x + 1)^{si}
		\right)
	\geq v_2((d^x+1)^s)
	=s v_2(d^x+1)
	=s v_2(d+1),
\end{equation*}
where for the last equality we have applied Lemma~\ref{lem:v2_oddexpo} and the fact that $x$ is odd.
Since $y$ is odd as well, for the right hand side of \eqref{eq:48rew} we have
\begin{equation*}
	v_2(d^y+1)
	=v_2(d+1).
\end{equation*}
Comparing the valuations of both sides of \eqref{eq:48rew}, we thus obtain
\[
	s v_2(d+1)
	\leq v_2(d+1),
\]
which is a contradiction for $s \geq 2$, since $v_2(d+1)>0$.
\end{proof}

\begin{newprop}\label{prop:eq50}
The equation
\begin{equation}\label{eq:50}
	(d^x + 1)^s - (d^y - 1)^m = 2
\end{equation}
has no solution in integers $d,x,y,s,m$ with $x,y \geq 1$ and ${d,s,m} \geq 2$.
\end{newprop}
\begin{proof}
Let $d,x,y,s,m$ be integers with $x,y \geq 1$, ${d,s,m} \geq 2$ satisfying ~\eqref{eq:50}.
We may again assume that $d$ is even, since otherwise ${s, m} \geq 2$ leads to a contradiction modulo $4$.

If $m$ is even, then Theorem~\ref{thm:Nagell} implies either $d^y - 1 = 1$ and $(d^x+1)^s = 3$, or $(d^y - 1)^m = 5^2$ and $(d^x + 1)^s = 3^3$. It is easy to check that neither is possible. Thus we have $m$ odd. We distinguish between three cases.

\caselI{Case 1:} $x = y$. 
Then we have
\begin{equation}\label{eq:x_equal_y}
	(d^x + 1)^s - (d^x - 1)^m = 2.
\end{equation}
If $d^x - 1 \leq 1$, then we must have $d=2$ and $x = 1$, which does not lead to a solution as $(d^x + 1)^s \geq 3^2 = 9$.
Thus we have $d^x - 1 \geq 2$ and $d^x +1 \geq 2$. Since the equation $(d^x + 1)^\alpha - (d^x - 1)^\beta = 2$ has solutions with exponents $(\alpha,\beta)=(1,1)$ and $(s,m)$, this contradicts Theorem~\ref{thm:bennett2001-t2}.

\caselI{Case 2:} $y<x$. Since $m$ is odd, using the binomial theorem, we can rewrite \eqref{eq:50} as
\begin{equation}\label{eq:50_m_odd}
	\sum_{i=1}^s {s \choose i} d^{xi}
	= \sum_{j=1}^m {m \choose j} (-1)^{m - j} d^{yj}.
\end{equation}
Since $1\leq y < x$, this implies that $d \mid {m \choose 1} = m$, which is a contradiction to $d$ being even and $m$ being odd.

\caselI{Case 3:} $y>x$.
By an analogous argument to that  in Case 2, we have that $d \mid s$, and, in particular, that $s$ is even. 
We may therefore apply Proposition~\ref{prop:main} to our main equation \eqref{eq:50}, to conclude that $m \leq 1237$ and  $d^y -1 \equiv -1 \pmod{Q}$, where $Q$ is the product of primes defined in Proposition~\ref{prop:main}.
The congruence implies 
$ Q \mid d^y$ and therefore $Q \mid d$. Since $d \mid s$, we have  in particular that
$Q \mid s$, and so $s > 10^{102}$.
In other words, we have that, in equation~\eqref{eq:50}, $s$ is extremely large, while $m$ is rather small.
Now note that
\begin{equation}\label{eq:prep_xsym}
	2 
	= (d^x + 1)^s - (d^y - 1)^m
	> d^{xs} - d^{ym},
\end{equation}
whence
\begin{equation}\label{eq:xsym}
	xs \leq ym.
\end{equation}
In order to obtain a contradiction, we consider the $2$-adic valuation of the expressions in equation \eqref{eq:50}. We rewrite this as 
\begin{equation}\label{eq:50rewr}
	(d^x + 1)^s - 1 = (d^y - 1)^m + 1.
\end{equation}
Since $d^y - 1 \geq 1$ and $m$ is odd, by Lemma~\ref{lem:v2_oddexpo}, the $2$-adic valuation of the right hand side of \eqref{eq:50rewr} is 
\[
 	v_2((d^y - 1)^m + 1) 
 	= v_2(d^y)
 	= y v_2(d).
\]
For the left hand side, we obtain from Lemma~\ref{lem:v2_powers-1} that
\begin{align*}
	v_2((d^x + 1)^s - 1)
	&\leq v_2((d^x+1)^2 - 1) + v_2(s) - 1\\
	&= v_2(d^x) + v_2(d^x + 2) + v_2(s) - 1\\
	&\leq \frac{\log (d^x + 2)}{\log 2} + \frac{\log s}{\log 2}
	\leq \frac{x \log d + \log 2 + \log s}{\log 2}.
\end{align*}
Moreover, recall that $d\mid s$, so $d\leq s$. We estimate further:
\[
	v_2((d^x + 1)^s - 1)
	\leq \frac{x \log s + \log 2 + \log s}{\log 2}
	\leq \frac{2x \log s + \log 2}{\log 2}
	< 8 x \log s.
\]
It follows then that
\[
	y v_2(d)
	\leq 8 x \log s.
\]
Combining this with \eqref{eq:xsym}, we obtain
\[
	xs 
	\leq m y
	\leq m \cdot y v_2(d)
	\leq m \cdot  8 x \log s.
\]
This implies
\[
	s \leq 8m \log s,
\]
which is impossible for $m\leq 1237$ and $s >10^{102}$.
\end{proof}

\begin{newprop}\label{prop:eq49}
The equation
\begin{equation}\label{eq:49}
	d^y - ((d^x - 1)^s - 1)^t = 2
\end{equation}
has no solution in integers $d,x,y,s,t$ with $x\geq 1$ and ${d,y,s,t} \geq 2$.
\end{newprop}
\begin{proof}
Let $d,x,y,s,t$ be integers with $x\geq 1$, ${d,y,s,t} \geq 2$,  satisfying equation \eqref{eq:49}.
Yet again, we may assume that $d$ is odd, since otherwise ${y, t} \geq 2$ leads to a contradiction  modulo $4$.
If $t$ is even, then Theorem~\ref{thm:Nagell} implies $d=3$ and $(d^x - 1)^s - 1 = 5$, which is impossible. Thus we have $t$ odd.

We distinguish between two cases, corresponding to the parity of $y$.

\caselI{Case 1:} $y$ is odd.
Using the binomial theorem, we rewrite \eqref{eq:49} as
\begin{equation}\label{eq:49_todd}
	d^y - 1
	= \sum_{i=1}^t {t \choose i}(-1)^{t-i} (d^x-1)^{si},
\end{equation}
which implies $d^x-1 \mid d^y - 1$. From Lemma~\ref{lem:qx-1_qy-1} it follows that $x\mid y$, and in particular $x$ is odd as well.

Now we compare the 2-adic valuations of the expressions in equation~\eqref{eq:49_todd}. Since $d$ and $y$ are odd, for the left hand side we have by Lemma~\ref{lem:v2_powers-1} that
\[
	v_2(d^y - 1)
	= v_2(d - 1).
\]
On the other hand, the right hand side of \eqref{eq:49_todd} is equal to $((d^x-1)^s-1)^t + 1$. We recall that $t$ is odd and apply Lemma~\ref{lem:v2_oddexpo}, then we recall that $d$ and $x$ are odd and apply Lemma~\ref{lem:v2_powers-1}. This way we obtain
\[
	v_2\left(
		((d^x-1)^s-1)^t + 1
		\right)
	= v_2((d^x-1)^s)
	= s v_2(d^x-1)
	= s v_2(d-1).
\]
Comparing both sides of \eqref{eq:49_todd}, we obtain
\[
	v_2(d - 1) = s v_2(d-1),
\]
which is a contradiction for $s \geq 2$, since $v_2(d - 1) > 0$.

\caselI{Case 2:} $y$ is even. 
Then Proposition~\ref{prop:main} applied to equation~\eqref{eq:49} implies that $t$ is prime with $41 \leq t \leq 1237$, $t \equiv 13,17,19,23 \pmod{24}$ and 
\[
	(d^x-1)^s-1 \equiv -1 \pmod{Q},
\]
where $Q$ is the product of primes defined in Proposition~\ref{prop:main}.
The last congruence implies 
\begin{equation}\label{eq:congr_dx}
	d^x \equiv 1 \pmod{Q}.
\end{equation}

Since equation \eqref{eq:49} implies that  $x \leq y$, reducing equation \eqref{eq:49} modulo $d^x$, we find that
\[
	0 - ((-1)^s - 1)^t \equiv 2 \pmod{d^x}.
\]
If $s$ is even, we obtain a contradiction (note that $d^x >2$, as $d$ is odd). We may thus suppose that  $s$ is odd, whence
\[
	- (-2)^t \equiv 2 \pmod{d^x}.
\]
This implies (since $t$ and $d$ are odd)
\[
	2^{t-1} \equiv 1 \pmod{d^x},
\]
and therefore 
\[
	d^x \mid 2^{t-1} - 1.
\]

Suppose first that $x = 1$. Then, setting $f = d-1$, equation \eqref{eq:49} becomes 
\[
	(f+1)^y - (f^s -1)^t = 2.
\]
Since $d$ is odd, we have $d\geq 3$ and therefore $f\geq 2$.  Proposition~\ref{prop:eq50} thus yields a contradiction.
We may therefore suppose that $x\geq 2$. 

There are only $102$ primes $t$ with $41 \leq t \leq 1237$, $t \equiv 13,17,19,23 \pmod{24}$. For each $t$ we look at the prime factorization of $2^{t-1} - 1$ (we use the known factorizations from the Factoring Database \cite{factordb}). Then we check all divisors of the shape $d^x$ with $x\geq 2$.
They are all rather small, and none of them satisfies \eqref{eq:congr_dx}. In fact, the only factor $d^x$ with $x\geq 2$ and $d^x \geq 10^5$ appears for $t = 1093$, where we have 
\[
2^{1093-1} = 3^3 \cdot 7^2 \cdot 13^2 \cdot 1093^2 \cdot t_0,
\]
with $t_0$ squarefree. In this case,
$$
d^x= (3\cdot 7 \cdot 13 \cdot 1093)^2
$$
is still much too small  to satisfy \eqref{eq:congr_dx}.
Relevant factorizations and code can be found at: 
\url{https://cocalc.com/IngridVukusic/ConsecutiveTriples/Factorizations}

\noindent This completes the proof of Proposition \ref{prop:eq49}.

\end{proof}

\section{Proof of Theorem \ref{thm:3x2md}}\label{sec:proof_3x2md}

Assume that $a,b,c\geq 2$ are distinct integers such that $(a,b,c)$, $(a+1,b+1,c+1)$ and $(a+2,b+2,c+2)$ are each 2-multiplicatively dependent. We wish to derive a contradiction.
Our proof of this consists of a number of cases,  split  into three parts. In the first part, we assume that $\{2,8\}\subset\{a,b,c\}$. In the second part, we suppose that  $8\in \{a,a+1, b, b+1, c, c+1\}$. In the third part, we treat all other cases.

\subsection{Part 1}
Assume that $\{2,8\}\subset\{a,b,c\}$. In this part of the proof, we do not assume that $a,b,c$ are ordered. Thus, we may assume, without loss of generality, that $a=2$ and $b=8$. Then $(a,b,c)=(2,8,c)$ and $(a+1,b+1,c+1)=(3,9,c+1)$ are indeed each 2-multiplicatively dependent. The triple $(a+2,b+2,c+2)=(4,10,c+2)$ is 2-multiplicatively dependent\ if and only if $c$ is of the shape $2^x - 2$ or $10^x -2$. This exactly corresponds to the ordered triples from \eqref{eq:3x2-md} in Theorem~\ref{thm:3x2md}.

\subsection{Part 2}
Assume that $8\in \{a,a+1, b, b+1, c, c+1\}$. In this part of the proof we again do not assume that $a,b,c$ are ordered.
Therefore, we may assume, without loss of generality, that either $a=8$ or $a+1=8$. 

\caselI{Case a:} $a=8$.
Let us assume that $b,c \neq 2$, as those cases have been treated in Part 1 already. Now the triple $(a,b,c)=(8,b,c)$ is 2-multiplicatively dependent. Since the numbers $a,b,c$ are not ordered, we have, without loss of generality,  only two subcases, corresponding to $(a,b)$ or $(b,c)$, respectively, being multiplicatively dependent.  

\caselII{Case a.1:} $(a,b)=(8,b)$ is multiplicatively dependent, i.e.\ $b= 2^x$ for some $x>1$, $x\neq 3$. Since $(a+1,b+1,c+1)$ is also 2-multiplicatively dependent, and, via Lemma~\ref{lem:consecutivePairs}, $(a+1,b+1)$ cannot be multiplicatively dependent, it follows that either $(a+1,c+1)$ or $(b+1,c+1)$ is multiplicatively dependent.

\caselIII{Case a.1.1:} $(a+1, c+1) = (9, c+1)$ is multiplicatively dependent, i.e.\ $c = 3^y -1$ with $y>2$. Since $(a+2,b+2,c+2)$ is 2-multiplicatively dependent\ and both $(a,b)$ and $(a+1,c+1)$ are multiplicatively dependent, by Lemmata \ref{lem:consecutivePairs} and \ref{lem:consecutivePairs2}, we must have that $(b+2,c+2) = (2^x + 2, 3^y + 1)$ is multiplicatively dependent, i.e.\ $(2^x + 2, 3^y + 1) = (q^s, q^t)$ for some positive integers $s\neq t$ and $q>1$. Since $y>2$, the equation $3^y + 1 = q^t$ implies by Theorem~\ref{thm:Catalan} that $t = 1$. Then we have 
\begin{equation*}
	2^x + 2
	= q^s 	
	= (3^y + 1)^s.
\end{equation*}
Rewriting this as
\[
	((2^1 + 1)^y + 1)^s - 2^x = 2,
\]
this contradicts Proposition~\ref{prop:eq48}.

\caselIII{Case a.1.2:} $(b+1,c+1) = (2^x + 1, c+1)$ is multiplicatively dependent.
Since $(a+2,b+2,c+2)$ is 2-multiplicatively dependent\ and both $(a,b)$ and $(b+1,c+1)$ are multiplicatively dependent, by Lemmata \ref{lem:consecutivePairs} and \ref{lem:consecutivePairs2}, we must have that $(a+2,c+2) = (10, c+2)$ is multiplicatively dependent, i.e.\ $c=10^y - 2$ with $y>1$. Then since $(b+1,c+1)=(2^x + 1, 10^y - 1)$ is multiplicatively dependent, we have $(b+1,c+1) = (2^x + 1, 10^y - 1) = (q^s, q^t)$ for some positive integers $s\neq t$ and $q>1$. On the one hand, since $x\neq 1,3$, the equation $2^x + 1 = q^s$ by Theorem~\ref{thm:Catalan} implies $s = 1$. Since $y>1$, the equation $10^y - 1 = q^t$ implies by Theorem~\ref{thm:Catalan} that $t=1$. Thus $t = s$, a contradiction.

\caselII{Case a.2:} $(b, c)$ is multiplicatively dependent\ and we still have $a=8$.  Then we have $(b,c) = (q^x, q^y)$ for some positive integers $x\neq y$ and $q>1$. 

Next, $(a+1,b+1,c+1)$ is 2-multiplicatively dependent, so by Lemma~\ref{lem:consecutivePairs} either $(a+1,b+1)$ is multiplicatively dependent, or $(a+1,c+1)$ is multiplicatively dependent. Since the numbers are not ordered, we may assume, without loss of generality, that $(a+1,b+1)=(9, q^x + 1)$ is multiplicatively dependent. Then we have $b+ 1 = q^x + 1 = 3^s$ for some $s >2$. But then Theorem~\ref{thm:Catalan} implies that $x = 1$ and thus $c = q^y = b^y = (3^s - 1)^y$.

Finally, since $(a+2,b+2,c+2)$ is 2-multiplicatively dependent\ and both $(b,c)$ and $(a+1,b+1)$ are multiplicatively dependent, by Lemmata \ref{lem:consecutivePairs} and \ref{lem:consecutivePairs2}, we must have that $(a+2, c+2) = (10, q^y + 2)$ is multiplicatively dependent.
This implies $c+2 = q^y + 2 = 10^t$ for some $t>1$. 

Summarizing, we find that
\[
	10^t	
	= c + 2
	= q^y + 2
	=  (3^s - 1)^y + 2.
\]	
Since $t>1$, the left hand side of the above chain of equations is divisible by 4. On the other hand, since $y>1$,
the right hand side is congruent 2 modulo 4, a contradiction.

\caselI{Case b:} $a = 7$. 
We proceed analogously to Case a. Again we have two subcases, depending on which of  $(a,b)$ or $(b,c)$ is multiplicatively dependent.

\caselII{Case b.1:} $(a,b) = (7,b)$ is multiplicatively dependent, i.e. $b = 7^x$ for some $x>1$.
And again, we have two subcases, namely $(a+1, c+1)$ is multiplicatively dependent or $(b+1,c+1)$ is multiplicatively dependent.

\caselIII{Case b.1.1:} $(a+1,c+1) = (8,c+1)$ is multiplicatively dependent, i.e. $c = 2^y -1$ for some $y>1$, $y\neq 3$. Then analogously to Case a.1.1 we have that $(b+2,c+2) = (7^x + 2, 2^y + 1)$ is multiplicatively dependent, i.e.\ $(7^x + 2, 2^y + 1) = (q^s,q^t)$ for some positive integers $s\neq t$. Since $y>1$, $y\neq 3$, the equation $2^y + 1 = q^t$ implies by Theorem~\ref{thm:Catalan} that $t=1$. Then we have
\begin{equation*}
	7^x + 2
	= q^s 
	= (2^y + 1)^s.
\end{equation*}
Rewriting this as
\[
	(2^y + 1)^s - (2^3 -1)^x = 2,
\]
we immediately get a contradiction from Proposition~\ref{prop:eq50}.

\caselIII{Case b.1.2:} $(b+1,c+1) = (7^x + 1, c+1)$ is multiplicatively dependent. Then, as in Case a.1.2 we have that in the next triple $(a+2, c+2) = (9, c+2)$ must be multiplicatively dependent, i.e.\ $c = 3^y - 2$ for some $y > 2$. Then since $(b+1,c+1) = (7^x + 1, 3^y - 1)$ is multiplicatively dependent, we have $(7^x + 1, 3^y - 1) = (q^s, q^t)$ for some positive integers $s \neq t$ and $q>1$. Then, since $x >1$ and $y>2$, Theorem~\ref{thm:Catalan} implies that $s = t = 1$, a contradiction.

\caselII{Case b.2:} $(b,c)$ is multiplicatively dependent\ and we still have $a = 7$. Then we have $(b,c) = (q^x, q^y)$ for some positive integers $x\neq y$ and $q>1$.  As in Case a.2, we may assume that $(a+1,b+1)=(8,b+1)$ and $(a+2,c+2)=(9,c+2)$ are each multiplicatively dependent. Thus we get $b = 2^s - 1$ for some $s>1$ and $s\neq 3$, and $c= 3^t - 2$ for some $t >2$. Now we have $(2^s - 1, 3^t - 2) = (b,c) = (q^x, q^y)$. Since $s>1$, $s\neq 3$, Theorem~\ref{thm:Catalan} implies $x=1$. Thus we obtain
\[
	3^t - 2
	= c
	= q^y
	= b^y
	= (2^s - 1)^y.
\]
Rewriting this as
\[
	(2^1 + 1)^t - (2^s - 1)^y = 2,
\]
we immediately reach a contradiction from Proposition~\ref{prop:eq50}.

\subsection{Part 3}
In Parts 1 and 2,  we have handled all cases where 
$$
\{8,9\} \subset \{a,a+1,a+2\}, \; \; \{8,9\} \subset \{b,b+1,b+2\} \; \; \mbox{ or } \; \; \{8,9\} \subset \{c,c+1,c+2\}. 
$$
In Part 3, we now assume that we are in none of those cases. Therefore, by Theorem~\ref{thm:Catalan}, for $d=a,b,c,$ we have that
\begin{equation}\label{eq:cond28}
\text{the set $\{d,d+1,d+2\}$ does not contain two consecutive perfect powers.}
\end{equation}
For the remainder of the proof, we assume that $1<a<b<c$.
From Lemmata \ref{lem:consecutivePairs} and \ref{lem:consecutivePairs2}, it follows that the triples must satisfy
\begin{align*}
	(a+i,b+i,c+i) 
	&= (p^s, p^t, c+i),\\
	(a+j,b+j,c+j)
	&= (q^u, b+j , q^v),\\
	(a+k,b+k,c+k)
	&= (a+k, r^x, r^y),
\end{align*}
where $\{i,j,k\}=\{0,1,2\}$ and $p,q,r,s,t,u,v,x,y$ are positive integers. Since $1<a<b<c$, we have $p,q,r>1$, $s<t$, $u<v$ and $x<y$, and in particular $t,v,y>1$. Therefore, $q^v$ and $r^y$ are both perfect powers and by \eqref{eq:cond28} we have
$q^v - r^y \neq \pm 1$, so $j-k\neq \pm 1$, i.e.\ $\{j,k\}=\{0,2\}$ and $i=1$. Thus $p^t-r^x = \pm 1$, and since $t>1$, this implies $x=1$. 

Also, $p^s - q^u =\pm 1$, so either $s=1$ or $u=1$. Now we consider 4 cases, according to whether $(j,k)=(0,2)$ or $(j,k)=(2,0)$, and whether $s=1$ or $u=1$.

\caselI{Case 1:} $(j,k)=(0,2)$ and $s=1$. Then we have
\begin{align*}
	(a,b,c)
	&= (q^u, b, q^v),\\	
	(a+1,b+1,c+1) 
	&= (p, p^t, c+1),\\
	(a+2,b+2,c+2)
	&= (a+2, r, r^y).
\end{align*}
Now we see that $p=q^u+1$ and $r=p^t+1$, so $r=(q^u+1)^t +1$ and therefore,
\[
	2=r^y-q^v =((q^u+1)^t +1)^y - q^v.
\]
Since ${q,t,y,v}\geq 2$, this is impossible by Proposition~\ref{prop:eq48}.

\caselI{Case 2:} 
$(j,k)=(0,2)$ and $u=1$. Then we have
\begin{align*}
	(a,b,c)
	&= (q, b, q^v),\\	
	(a+1,b+1,c+1) 
	&= (p^s, p^t, c+1),\\
	(a+2,b+2,c+2)
	&= (a+2, r, r^y),
\end{align*}
and we see that $q=p^s-1$ and $r=p^t+1$, so
\[
	2=r^y-q^v = (p^t+1)^y - (p^s-1)^v.
\]
Since ${p,y,v}\geq 2$, this is impossible by Proposition~\ref{prop:eq50}.

\caselI{Case 3:} 
$(j,k)=(2,0)$ and $s=1$. Then we have
\begin{align*}
	(a,b,c)
	&= (a, r, r^y),\\	
	(a+1,b+1,c+1) 
	&= (p, p^t, c+1),\\
	(a+2,b+2,c+2)
	&= (q^u, b+2, q^v),
\end{align*}
and we see that $p=q^u-1$ and $r=p^t-1=(q^u-1)^t-1$, so
\begin{equation*}
	2=q^v - r^y = q^v - ((q^u-1)^t-1)^y.
\end{equation*}
Since  ${q, v,t,y }\geq 2$, this is impossible by Proposition~\ref{prop:eq49}.

\caselI{Case 4:} 
$(j,k)=(2,0)$ and $u=1$. Then we have
\begin{align*}
	(a,b,c)
	&= (a, r, r^y),\\	
	(a+1,b+1,c+1) 
	&= (p^s, p^t, c+1),\\
	(a+2,b+2,c+2)
	&= (q, b+2 , q^v),
\end{align*}
and we see that $q=p^s+1$ and $r=p^t-1$, so
\begin{equation*}
	2=q^v - r^y = (p^s+1)^v - (p^t-1)^y.
\end{equation*}
Since ${p, v, y}\geq 2$, this is impossible by Proposition~\ref{prop:eq50}.

This completes the proof of Theorem \ref{thm:3x2md}.

\bibliographystyle{habbrv}
\bibliography{refs}

\begin{thebibliography}{10}
\expandafter\ifx\csname url\endcsname\relax
  \def\url#1{\texttt{#1}}\fi
\expandafter\ifx\csname doi\endcsname\relax
  \def\doi#1{\burlalt{doi:#1}{http://dx.doi.org/#1}}\fi
\expandafter\ifx\csname urlprefix\endcsname\relax\def\urlprefix{URL: }\fi
\expandafter\ifx\csname href\endcsname\relax
  \def\href#1#2{#2}\fi
\expandafter\ifx\csname burlalt\endcsname\relax
  \def\burlalt#1#2{\href{#2}{#1}}\fi

\bibitem{Bennett2001}
M.~A. Bennett.
\newblock On some exponential equations of {S}. {S}. {P}illai.
\newblock {\em Canad. J. Math.}, 53(5):897--922, 2001.
\newblock \doi{10.4153/CJM-2001-036-6}.

\bibitem{Bennett2004}
M.~A. Bennett.
\newblock Products of consecutive integers.
\newblock {\em Bull. London Math. Soc.}, 36(5):683--694, 2004.
\newblock \doi{10.1112/S0024609304003480}.

\bibitem{BennettSkinner2004}
M.~A. Bennett and C.~M. Skinner.
\newblock Ternary {D}iophantine equations via {G}alois representations and
  modular forms.
\newblock {\em Canad. J. Math.}, 56(1):23--54, 2004.
\newblock \doi{10.4153/CJM-2004-002-2}.

\bibitem{BugeaudMignotteSiksek2006}
Y.~Bugeaud, M.~Mignotte, and S.~Siksek.
\newblock Classical and modular approaches to exponential {D}iophantine
  equations. {II}. {T}he {L}ebesgue-{N}agell equation.
\newblock {\em Compos. Math.}, 142(1):31--62, 2006.
\newblock \doi{10.1112/S0010437X05001739}.

\bibitem{Chen2012}
I.~Chen.
\newblock On the equations {$a^2-2b^6=c^p$} and {$a^2-2=c^p$}.
\newblock {\em LMS J. Comput. Math.}, 15:158--171, 2012.
\newblock \doi{10.1112/S146115701200006X}.

\bibitem{Cohen2007I}
H.~Cohen.
\newblock {\em Number theory. {V}ol. {I}. {T}ools and {D}iophantine equations},
  volume 239 of {\em Graduate Texts in Mathematics}.
\newblock Springer, New York, 2007.
\newblock \doi{10.1007/978-0-387-49923-9}.

\bibitem{Cohen2007II}
H.~Cohen.
\newblock {\em Number theory. {V}ol. {II}. {A}nalytic and modern tools}, volume
  240 of {\em Graduate Texts in Mathematics}.
\newblock Springer, New York, 2007.
\newblock \doi{10.1007/978-0-387-49894-2}.

\bibitem{Cohn1991}
J.~H.~E. Cohn.
\newblock The {D}iophantine equations {$x^3=Ny^2\pm 1$}.
\newblock {\em Quart. J. Math. Oxford Ser. (2)}, 42(165):27--30, 1991.
\newblock \doi{10.1093/qmath/42.1.27}.

\bibitem{Laurent2008}
M.~Laurent.
\newblock Linear forms in two logarithms and interpolation determinants. {II}.
\newblock {\em Acta Arith.}, 133(4):325--348, 2008.
\newblock \doi{10.4064/aa133-4-3}.

\bibitem{LeVeque1952}
W.~J. LeVeque.
\newblock On the equation {$a^x-b^y=1$}.
\newblock {\em Amer. J. Math.}, 74:325--331, 1952.
\newblock \doi{10.2307/2371997}.

\bibitem{Ljunggren1942}
W.~Ljunggren.
\newblock Zur {T}heorie der {G}leichung {$x^2+1=Dy^4$}.
\newblock {\em Avh. Norske Vid.-Akad. Oslo I}, 1942(5):27, 1942.

\bibitem{Matveev2000}
E.~M. Matveev.
\newblock An explicit lower bound for a homogeneous rational linear form in the
  logarithms of algebraic numbers. {II}.
\newblock {\em Izv. Math.}, 64(6):1217--1269, 2000.
\newblock \doi{10.1070/im2000v064n06abeh000314}.

\bibitem{Mihailescu2004}
P.~Mih\u{a}ilescu.
\newblock Primary cyclotomic units and a proof of {C}atalan's conjecture.
\newblock {\em J. Reine Angew. Math.}, 572:167--195, 2004.
\newblock \doi{10.1515/crll.2004.048}.

\bibitem{Nagell1954}
T.~Nagell.
\newblock Verallgemeinerung eines {F}ermatschen {S}atzes.
\newblock {\em Arch. Math. (Basel)}, 5:153--159, 1954.
\newblock \doi{10.1007/BF01899332}.

\bibitem{PappalardiShaShparlinskiStewart2018}
F.~Pappalardi, M.~Sha, I.~E. Shparlinski, and C.~L. Stewart.
\newblock On multiplicatively dependent vectors of algebraic numbers.
\newblock {\em Trans. Amer. Math. Soc.}, 370(9):6221--6244, 2018.
\newblock \doi{10.1090/tran/7115}.

\bibitem{Smart1998}
N.~P. Smart.
\newblock {\em {The Algorithmic Resolution of Diophantine Equations}}.
\newblock London Mathematical Society Studen Texts 41. Cambridge University
  Press, 1998.

\bibitem{Stormer1899}
C.~St\"{o}rmer.
\newblock Solution compl\`ete en nombres entiers de l'\'{e}quation
  {$m\arctan\frac1x+n\arctan\frac1y=k\frac{\pi}4$}.
\newblock {\em Bull. Soc. Math. France}, 27:160--170, 1899.
\newblock \urlprefix\url{http://www.numdam.org/item?id=BSMF_1899__27__160_1}.

\bibitem{factordb}
M.~Tervooren.
\newblock {\em Factoring Database}.
\newblock \urlprefix\url{http://factordb.com}.
\newblock Accessed: 2022--10--10.

\bibitem{sagemath}
{The Sage Developers}.
\newblock {\em {S}ageMath, the {S}age {M}athematics {S}oftware {S}ystem
  ({V}ersion 9.2)}, 2021.
\newblock \urlprefix\url{https://www.sagemath.org}.

\bibitem{VukusicZiegler2021}
I.~Vukusic and V.~Ziegler.
\newblock Consecutive tuples of multiplicatively dependent integers.
\newblock {\em J. Number Theory}, 2021.
\newblock \doi{https://doi.org/10.1016/j.jnt.2021.07.021}.

\bibitem{Zsigmondy1892}
K.~Zsigmondy.
\newblock Zur {T}heorie der {P}otenzreste.
\newblock {\em Monatsh. Math. Phys.}, 3(1):265--284, 1892.
\newblock \doi{10.1007/BF01692444}.

\end{thebibliography}

\end{document}